\documentclass[11pt,a4paper]{amsart}

\usepackage[utf8]{inputenc}
\usepackage{amsmath,amsfonts,amssymb,amsthm,mathrsfs}
\usepackage{todonotes} 
\usepackage{hyperref, enumitem}
\newcommand{\N}{\mathbb{N}}
\newcommand{\Z}{\mathbb{Z}}

\newcommand{\R}{\mathbb{R}}

\newcommand{\T}{\mathbb{T}}

\newcommand{\cC}{\mathcal{C}}
\newcommand{\cF}{\mathcal{F}}
\newcommand{\cS}{\mathcal{S}}
\newcommand{\diff}{ \; \mathrm{d}}

\newcommand{\supp}{\operatorname{supp}}

\newcommand{\pv}{\mathrm{pv}}
\renewcommand{\div}{\operatorname{div}}

\newcommand{\di}[1]{\Delta_{#1}}

\newcommand{\with}{\quad\hbox{with}\quad}
\newcommand{\andf}{\quad\hbox{and}\quad}

\def\ds{\delta\!\sigma}
\def\du{\delta\!u}
\def\dU{\delta\!U}

\newtheorem{theo}{\bf Theorem}[section]

\newtheorem{coro}{\bf Corollary}[section]
\newtheorem{lem}{\bf Lemma}[section]

\newtheorem{prop}{\bf Proposition}[section]

\title[Fractional Euler alignment system]{Regular solutions to the fractional Euler alignment system in the Besov spaces framework}
\author{Rapha\"el Danchin, Piotr B. Mucha, Jan Peszek, Bartosz Wr\' oblewski}
\address[R. Danchin]{Universit\'{e} Paris-Est,  LAMA (UMR 8050), UPEMLV, UPEC, CNRS,
 61 avenue du G\'{e}n\'{e}ral de Gaulle, 94010 Cr\'{e}teil Cedex, France.} \email{raphael.danchin@u-pec.fr}
\address[P.B. Mucha]{Instytut Matematyki Stosowanej i Mechaniki,
 Uniwersytet Wars\-zawski, 
ul. Banacha 2,  02-097 Warszawa, Poland.} 
\email{p.mucha@mimuw.edu.pl}

\address[J. Peszek]{Center for Scientific Computation and Mathematical Modeling (CSCAMM) and Department of Mathematics,
University of Maryland, College Park, MD 20742-4015, USA\newline
Instytut Matematyczny Polskiej Akademii Nauk, ul. \' Sniadeckich 8
00-656 Warszawa, Poland.}
\email{j.peszek@mimuw.edu.pl}
\address[B. Wr\' oblewski]{Instytut Matematyczny, Uniwersytet Wroc\l{}awski, 
pl. Grunwaldzki 2/4, 50-384 Wroc\l{}aw, Poland}

\date{\today}

\begin{document}

\begin{abstract} 
We here  construct (large) local and small global-in-time regular unique solutions to the fractional Euler alignment system  in the whole space $\R^d$, in the case where  the  deviation of 
the initial density from a constant  is sufficiently small. 
Our analysis strongly relies on the use of Besov spaces of the type $L^1(0,T;\dot B^s_{p,1}),$
which allow to get time independent estimates for the density even though 
it satisfies a transport equation with no damping. 
 Our choice of a functional setting is not optimal but aims at providing a transparent and accessible argumentation.
\end{abstract}

\maketitle

{\it MS Classifications: } 35Q31, 35B65, 35R11, 76N10, 82C22.


\section{Introduction}
Collective dynamics of interactive particles leading to an emergent phenomenon is an increasingly popular subject of research 
with a number of applications ranging from biology or robotics to social sciences  \cite{perea, oscilsyn}. 
The common  trait of such models is that  relatively simplistic agents basing their behavior on limited information produce a complex structure like in 
 e.g. anthills or specific formations of birds. Interestingly, this  may be also observed in 
   seemingly more sophisticated phenomena, such as emergence of languages in primitive cultures \cite{lang}, distribution of goods \cite{goods} or gang-related crime \cite{crime}. 

From the mathematical viewpoint, these models are a source of many challenging problems
like deriving  the kinetic and then hydrodynamic models from basic ODE systems. 
One may mention e.g. the  well-known  Hilbert's sixth problem of axiomatization of mathematical physics, that 
consists in  providing a mathematically rigorous derivation of Boltzmann and Euler equations from the Newtonian particle systems.
\medbreak
The model of interactive particles  we aim at considering here is the following  Euler-type  hydrodynamic version 
of the  Cucker-Smale (CS) flocking model introduced in \cite{cuc1}:
\begin{align}\label{sys}
\left\{
\begin{array}{l}
\displaystyle \rho_t + \div ( \rho u  ) = 0,\\[10pt]
\displaystyle \rho  u_t + \rho (u \cdot \nabla) u \displaystyle  = \int_{\R^d} \biggl(\frac{u(t,y) - u(t,x)}{|y-x|^{d+ \alpha}}\biggr)\rho(t,x) \rho (t,y) \diff y \ .
\end{array}
 \right.
\end{align}
We refer to \eqref{sys} as the {\it fractional Euler flocking system}. Here $\rho(t,x)\in \R_+$ denotes the density of particles at position $x\in \R^d$ and time $t>0$. The vector field $u(t,x)\in\R^d$ represents 
the velocity of a particle that occupies the position $x\in\R^d$ at time $t\in\R_+$. The exponent $\alpha$ is assumed to be  in the {\mbox range $(1,2)$}.

We further suppose that the velocity tends to $0$ at infinity, and that the density goes to some positive constant  (say $1$ with no loss of generality), in a sense that will made clear once we will have introduced 
our functional setting (see below). 
\medbreak
Before we proceed, let us elaborate on the origin of the system. Going back to the ODE's and kinetic theory we begin with the CS model governed by the following system of $N$ particles
\begin{align}\label{csp}
\left\{
\begin{array}{lll}
\displaystyle\frac{d}{dt}x_i &=& v_i,\\
\displaystyle\frac{d}{dt}v_i &=& \displaystyle\frac{1}{N}\displaystyle\sum_{j=1}^N (v_j-v_i)\psi(|x_i-x_j|),
\end{array}
\right.
\end{align}
where $x_i(t)$ and $v_i(t)$ denote the position and velocity of the  $i$-th particle at time $t$. The function $\psi$, referred to as the {\it communication weight}, is usually non-negative, non-increasing 
and Lipschitz continuous.

 The CS model is  a simple example of  interacting particles' models that has been extensively studied in various, mostly qualitative, directions  such as collision avoidance \cite{ahn1, cuc2, ccmp} 
 and asymptotic and pattern formation \cite{car, ha1}. These two directions lead to further branching of research into the CS model with singular communication weight \cite{jpe, jps}, or various additional 
 forces that ensure specific asymptotic pattern formation \cite{hahaki, ha2} or under a leadership of selected individuals \cite{cuc3, shen}.
\smallbreak
Further, taking $N\to\infty$ in \eqref{csp}, we formally obtain  the kinetic equation
\begin{align}\label{csk}
f_t + v\cdot \nabla_x f + \div_v[F(f)f] &= 0,\\
\hbox{with }\ F(f)(t,x,v) &:= \int_{\R^d\times \R^d}(w-v)\psi(|x-y|)f(t,y,w)\diff y \diff w,\nonumber
\end{align}
where  $f(t,x,v)$ stands for  the distribution of the particles that at time $t$ have position $x$ and velocity $v$. 

The mathematical  derivation of \eqref{csk} is a challenging problem that has been considered  in e.g. \cite{haliu, hatad}. Particularly interesting from the point of view of this paper is the case of singular 
communication weight $\psi(s)=s^{-\gamma}$ which was studied in \cite{mp, jab}.
\smallbreak
Finally, taking the   hydrodynamical limit
\begin{equation*}
 f(t,x,v)=\rho(t,x) \delta_{u(t,x)}
\end{equation*}
in  \eqref{csk} and integrating over $\R^d_v,$ we find out  the continuity equation \eqref{sys}$_1$. As for the momentum equation $\eqref{sys}_2$,  it is obtained by testing with $v \rho(t,x) \delta_{u(t,x)}$. 
\smallbreak
The mathematically rigorous derivation of \eqref{sys} as a hydrodynamic limit of \eqref{csk}  has not yet been solved in full generality (see e.g. \cite{spoy} for recent developments in this direction). 
It is worthwhile noting  that the cases of regular and singular communication weights are significantly different.
In the present paper, we consider the singular communication weight $\psi(s)=s^{-\gamma}$ with  $\gamma = d + \alpha$ and $\alpha\in(1,2).$
Let us point out that \eqref{sys} can be  alternatively  counted as an element of non-classical hydrodynamics related to the description of phenomena of aggregation, flocking, and in general, 
of modeling of collective dynamics of interacting particles. It may be seen as 
a coupling between the classical continuity equation and a nonlinear parabolic system of fractional order, 
and analyzed by means of techniques that are barely related to the kinetic origin of the system.
\medbreak
Let us now bring the reader up to date with the research on system \eqref{sys}. Precisely this model but in the periodic 1D case has been investigated by Do {\it et al.} in \cite{kis} and simultaneously by Shvydkoy and Tadmor in \cite{tad1, tad2, tad4}. The main advantage of dimension $1$ is that  
  $u(t)$ and $\rho(t)$ are real numbers, and can thus be directly compared. Indeed, an easy  calculation then reveals that quantity
\begin{align*}
e(t,x):= u_x(t,x) + \int_{\T}\psi(|x-y|)(\rho(t,x)-\rho(t,y))\,dy
\end{align*}
satisfies
\begin{align*}
e_t + (ue)_x = 0.
\end{align*}
Thanks to these relations, one can  compare the regularity of $u$ and $\rho$ and, using the compactness of the 1D torus, obtain  a global-in-time positive lower bound on the density, thanks to which  the ``good" term on the right-hand side of \eqref{sys}$_2$ does not disappear. That method unfortunately
fails in higher dimension. 
Another approach is used in \cite{tad3, tan}, where the 1D and 2D models with compact initial data are studied by the method of characteristic. 
\medbreak
Our goal here is to provide a general  existence result in any dimension $d\geq 1,$ 
and to connect this class of problems to the well-developed language of Besov spaces.
Our main result  reads:
\begin{theo}\label{main-intro}
Assume that  $\alpha \in (1,2)$ and consider initial data $(\rho_0,u_0)$ so that $u_0$ and $\nabla u_0$ are  in 
$\dot B^{2-\alpha}_{d,1},$ and $\rho_0 -1$ and $\nabla\rho_0$ are  in $\dot B^{1}_{d,1}.$ 
 There exists $\varepsilon>0$ such that  if  in addition 
\begin{equation}\label{eq:smallness}
\|u_0\|_{\dot B^{2-\alpha}_{d,1}} +  \|\rho_0-1 \|_{\dot B^{1}_{d,1}}  <\varepsilon,
\end{equation}
then the fractional Euler system $\eqref{sys}$
has a unique global solution $(\rho,u)$ such that
$$u,\nabla u \in \cC_b(\R_+; \dot{B}^{2-\alpha}_{d,1}) \cap L^1(\R_+;\dot B^{2}_{d,1})
\andf(\rho -1),\nabla\rho \in \cC_b(\R_+; \dot B^{1}_{d,1}).$$
In the case where the smallness condition is fulfilled only by $\rho_0$ then there exist 
a unique solution $(\rho,u)$ on some time interval $[0,T]$ with $T>0$ so that
$$u,\nabla u \in \cC_b([0,T]; \dot{B}^{2-\alpha}_{d,1}) \cap L^1([0,T];\dot B^{2}_{d,1})
\andf(\rho -1),\nabla\rho \in \cC_b([0,T]; \dot B^{1}_{d,1}).$$
\end{theo}
Let us point out that the above result is not quite  optimal as regards regularity
assumptions. Indeed, the reader may check that System \eqref{sys} is invariant
for all $\lambda>0$ by the rescaling 
$$\rho(t,x)\leadsto \rho(\lambda^\alpha t,\lambda x)\quad\hbox{and}\quad
u(t,x)\leadsto \lambda^{\alpha-1} u(\lambda^\alpha t,\lambda x).$$
Optimal spaces for well-posedness are thus expected to have the above invariance. 
In the class of Besov spaces that we considered above, this would correspond to taking
initial data such that $(\rho_0-1,u_0)$ is only in $\dot B^1_{d,1}\times\dot B^{2-\alpha}_{d,1}.$ 
The smallness condition \eqref{eq:smallness} is thus at the critical level of 
regularity, but one more derivative is required on the data. 
The reason we choose to work with so much regularity is essentially to offer the reader an elementary proof with as less as possible technicalities. 
This choice gives us the possibility to obtain solutions in the class that naturally comes from the a priori estimate, and higher regularity allows to control uniqueness.
From the mathematical viewpoint the biggest challenge is to control the regularity of density for all time
(if \eqref{eq:smallness} is fulfilled). Roughly speaking, the transport theory
requires the velocity field to be at least in $L^1(0,\R_+;{\rm Lip})$, and it is guaranteed by regularity $L^1(\R_+;\dot B^2_{d,1})$ of the vector field obtained by analysis of $(\ref{sys})_2$. 
We derive estimates of solutions to (\ref{sys}) at this level of regularity, but they are insufficient to prove the uniqueness, owing to  the hyperbolic nature of the continuity equation.
 Here the troublemaker is the nonlinear nonlocal term
\begin{align*}
\int_{\R^d} \biggl(\frac{u(t,y) - u(t,x)}{|y-x|^{d+ \alpha}}\biggr)\rho(t,x) \rho (t,y) \diff y\ ,
\end{align*}
which, for too low  regularity, cannot be estimated properly.
\medbreak
 It is not clear whether one could find a suitable setting to avoid higher regularity like in \cite{hoff} or \cite{DM12}.
Here we chose to increase  the regularity of solutions by one derivative. That choice is 
the simplest one as it enables us to use essentially the same functional setting 
for the solution and its first order space derivatives. 
\medbreak
As a general remark we would like to underline that Theorem \ref{main-intro} can  be seen as  a first quantitative result for system (\ref{sys}), which allows to investigate 
interesting qualitative properties of solutions to the studied system. The basic energy balance for $\eqref{sys}$
\begin{align*}
\frac{d}{dt} \int_{\R^d}\rho |u|^2 dx + \int_{\R^d\times \R^d}\frac{|u(x)-u(y)|^2}{|x-y|^{d+\alpha}}\rho(x)\rho(y)\diff y \diff x = 0
\end{align*}
leads, under the assumption $\rho\approx 1$, to the asymptotic decay of the velocity to $0$. However, since we are working in whole space, proving exponential decay is impossible
(in contrast with \cite{kis} that analysis the system with  periodic boundary conditions, 
in the one-dimensional case). The information coming from Theorem \ref{main-intro} allows to obtain the following information concerning the decay of the velocity at infinity, showing
the expected flocking (see the proof in Appendix).
\begin{coro}\label{coro1}
 Let $(\rho,u)$ be a global in time solution given by Theorem \ref{main-intro}. Then 
 $$
 \|u(t)\|_{L^\infty} \to 0 \mbox{ \ \ as \ \ } t \to \infty.
 $$
\end{coro}
\medbreak
System \eqref{sys} is  comparable to the classical compressible Navier-Stokes system. 
The main difference for us is the lack of an  `effective viscous flux' (like $-\div u + p(\rho)$ for the compressible
Navier-Stokes system) enabling us to glean some time-decay or compactness for  the density. 
Recall that exhibiting such a quantity is the cornerstone of the proof of global 
results both for weak solutions \cite{fei,lions,no-st} or  regular solutions \cite{Dan,MN,Mucha}.
 In the case of \eqref{sys}, the only way to control regularity of the density is through the velocity, whence the need  of the $L^1$ integrability in time of suitable norms of $u.$
 This is of course closely  connected with our choice of Besov spaces of the type $\dot B^s_{p,1}$.
\medbreak
The paper is organized as follows. In Section \ref{prel} we introduce the notation and basic definitions and tools related to Besov spaces. Then, in Section \ref{res}, we reformulate the main result
 and give an overview  of the proof. Finally, in Section \ref{prof}, we prove Theorem \ref{main-intro}. 
 For the reader convenience,  we include in the appendix basic existence results for the continuity and the linearized
 momentum equations in Besov spaces and the proof of Corollary \ref{coro1}.


\section{Preliminaries}\label{prel}

Let us shortly introduce the main notation of  the paper. By ${\mathcal S}$ we denote the Schwartz space and consequently ${\mathcal S}^{'}$ is the space of tempered distributions. 
 The Fourier transform of $u$ with respect to the space variable is denoted by ${\mathcal F}u = \widehat{u}$. We shall  use $\|\cdot\|_p$ to denote the norm in the space $L^p(\R^d).$ 
Finally, we use the abbreviated form $L^p_TX$ for the  space $L^p(0,T; X)$ and 
$L^pX$ means $L^p(\R_+;X).$  
 Throughout the paper the letter $C$ denotes a generic constant.
\medbreak

Let us next  introduce the \emph{fractional Laplacian} and \emph{homogeneous Besov spaces}.

\subsection*{Fractional Laplacian}
Let $u: \R^d \mapsto \R$ be a Schwartz function and $\alpha \in (0,2)$. Then, the fractional Laplacian of $u$ is given by
\begin{equation} \label{fraclap}
( - \Delta)^{\alpha /2} u (x) = -c_{d,\alpha} \ \pv \int_{\R^d} \frac{u(x) - u(y)}{|x-y|^{d+\alpha}} \diff y \ .
\end{equation}
Here $\pv$ stands for the principal value of integral and $c_{d,\alpha}$ is a dimensional constant 
that ensures that  for all  smooth enough function $u$ on $\R^d,$ 
we have\footnote{According to e.g.  \cite[formula (2.15)]{bucur},   we have
$\displaystyle{c_{d,\alpha} = \frac{2^{\alpha} \Gamma ( d/2 + \alpha/2 )}{\pi^{d/2} \Gamma (- \alpha /2 )}\cdotp}$}
$$
( - \Delta)^{\alpha/2} u= \mathcal F^{-1} \{ | \xi|^{\alpha} \widehat u ( \xi ) \}\cdotp
$$
As can be easily observed by looking on the Fourier side, operator 
$(-\Delta)^{\alpha/2}$ maps the subspace $\cS_0$  of $\cS$ functions
with Fourier transform supported away from the origin, to itself.

\subsection*{Besov spaces}
 Let $\chi$ be a smooth function compactly supported in the ball $B(0,\frac{4}{3}),$ and  set $\varphi(\xi):=\chi(\xi/2)-\chi(\xi)$ 
 so that $\varphi$ is  smooth, supported in the annulus $B(0,\frac{8}{3})\setminus B(0,\frac{3}{4})$ and fulfills 
 $$ \sum_{j \in \Z} \varphi ( 2^{-j}  \xi )=1\quad\hbox{on}\quad \R^d \setminus \{ 0 \}.$$ 
  Let us introduce
\begin{equation*}
h = \mathcal F^{-1} \varphi  \andf  \widetilde h = \mathcal F^{-1} \chi.
\end{equation*}
The homogeneous dyadic blocks $\dot \Delta_j$ are defined for all $j \in \Z$  by
\begin{equation*}
\dot \Delta_{j} u := 2^{jd} \int_{\R^d} h(2^j y) u(x-y) \diff y =\cF^{-1}\bigl(\varphi(2^{-j}\cdot)u\bigr).
\end{equation*} 
The homogeneous low-frequency cut-off operator $\dot S_j$ is defined by
\begin{equation*}
\dot S_j u := 2^{jd} \int_{\R^d} \widetilde h(2^j y) u(x-y) \diff y=\cF^{-1}\bigl(\chi(2^{-j}\cdot)u\bigr).
\end{equation*}
All of the above operators are bounded in $L^p$ with norms independent of $j$ and $p$.
\medbreak
For $ s \in \R$ and $p,q \in [1, \infty],$ we introduce the following homogeneous Besov (semi)-norms~:
\begin{equation*}
\| u \|_{ \dot B^s_{p,q}} : = \left\| 2^{js} \| \dot \Delta_j u\|_p\right\|_{\ell^q(\Z)}\cdotp
\end{equation*}
Then, following \cite{danchin}, we define the homogeneous Besov space  to be the set of tempered 
distributions $u$ so that 
$$\|u\|_{\dot B^s_{p,q}}<\infty\andf \lim_{j\to-\infty} \|\dot S_ju\|_\infty=0.$$
The low frequency condition guarantees that  $\dot B^s_{p,q}$ is a normed space.
\medbreak
In the case $s\in(0,1),$  equivalent  norms may be defined in terms 
of finite difference. More precisely,  for a given function $f: \R^d \mapsto \R^m$ and $y \in \R^d,$ let us denote 
\begin{equation}\label{delta}
\di{y} f(x) := f(x+y) - f(x).
\end{equation}
Then we have (see  e.g.  \cite[p. 74]{danchin})  for all  function $u$  in $\cS$,  
the following equivalence: 
\begin{align}\label{nbes2}
c\| u \|_{\dot B^s_{p,q}}\leq\left\| \frac{\|\di{y} u \|_p}{|y|^{d+s}} \right\|_q\leq C\| u \|_{\dot B^s_{p,q}},
\end{align}
where the positive constants $c$ and $C$ depend only on $s,$ $d$ and $p.$ 
\medbreak
As in the rest of the paper, we will only consider Besov spaces with finite $p,$ and $q=1,$ 
we just enumerate in the following lemma  the most important properties of those spaces 
(see \cite[Chap. 2]{danchin} for more details). 

%

\begin{lem}[Basic properties of homogeneous Besov spaces]\label{super}
\begin{enumerate}[label=(\alph*)] 
\item  The space $ \dot B^s_{p,1}$ is complete whenever  $s\leq d/p.$
\item \label{dense} If  $p$ is finite  then $\cS_0$ is dense in $\dot B^s_{p,1},$ and so does $\cS$ if $s>-d/p'.$  

\item
We have the continuous embedding $\dot B^s_{p,1}\hookrightarrow \dot B^{s-d/p+d/q}_{q,1}$ for
all $1\leq p\leq q\leq\infty,$ and the homogeneous Besov space $ \dot B^1_{d,1}$ is continuously embedded in the space of continuous functions vanishing at infinity.
\item For all $s_1<s_2$ and $\theta\in (0,1)$ we have
$\dot{B}^{s_1}_{p,1}\cap \dot B^{s_2}_{p,1}\subset \dot{B}^{\theta s_1 + (1-\theta)s_2}_{p,1}$ and
$$
\|u\|_{\dot{B}^{\theta s_1 + (1-\theta)s_2}_{p,1}}\leq \|u\|^\theta_{\dot{B}^{s_1}_{p,1}}\|u\|^{1-\theta}_{\dot{B}^{s_2}_{p,1}}.
$$

\item \label{besovineq}   For any $s\in\R$ and $p\in[1,\infty],$ the gradient operator maps 
$\dot B^{s+1}_{p,1}$ to $\dot B^{s}_{p,1}$.
\item For any $s\leq d/p$ and $p\in[1,\infty)$ the operator $(- \Delta)^{\alpha /2}$ maps   $\dot B^{s+\alpha}_{p,1}$
to  $\dot B^s_{p,1}.$
\item \label{prod}
The space $\dot B^1_{d,1}$ is stable by product, and
 there exists $C>0$ independent of $u$ and $v$ such that we have
\begin{equation*}
\| uv \|_{\dot B^{1}_{d,1}} \leq C \| u \|_{\dot B^{1}_{d,1}} \| v \|_{\dot B^{1}_{d,1}}.
\end{equation*}
\item For any $\alpha \in [1,2]$ and $\theta\in[0,\alpha-1],$  we have 
\begin{equation} \label{prodalpha}
\| uv \|_{\dot B^{2- \alpha}_{d,1}} \leq C \| u \|_{\dot B^{1-\theta}_{d,1}} \| v \|_{\dot B^{2+\theta- \alpha}_{d,1}} \,;
\end{equation}
\item \label{commutator}
Let $w$ be a vector field over $\R^d$. Define 
\begin{equation*}
R_j := [ w \cdot \nabla , \Delta_j]u =  w \cdot \nabla \dot \Delta_j u - \dot \Delta_j ( w \cdot \nabla u ) \ .
\end{equation*}
If $-1<s\leq2$ then there exists a positive constant $C$, independent of $u$ and $w$, such that 
\begin{equation*}
\| R_j \|_{\dot B^s_{d,1}} \leq C \| \nabla w \|_{\dot B^1_{d,1}} \|u \|_{\dot B^s_{d,1}} .
\end{equation*}
%
%
\end{enumerate}
\end{lem}

\begin{proof}
For the proofs of the above statements we redirect the reader to standard references regarding Besov spaces. More specifically, $(a)$, 
\ref{dense},  $(c)$,  $(d)$,    $(e-f)$ and $(g)$ follow from \cite{danchin}, Theorem 2.25,  Proposition 2.27, Proposition 2.39, 
Proposition 2.22,    Proposition  2.30 and Corollary 2.54, respectively. 
As for \ref{commutator}, it is a particular case of Lemma 2.100.
Finally, \eqref{prodalpha}  stems from the results of continuity of the remainder
and paraproduct operators that are stated in  \cite[Theorems 2.47 and 2.52]{danchin}
 (see also \cite[Lemma 1.5]{Dan-Fourier} for a more detailed proof). 
\end{proof}


\section{Renormalized system and  restatement of the main result}\label{res}

In order to properly define all elements of system \eqref{sys} in low regularity, we renormalize it  in three steps.
First, we rewrite $\eqref{sys}_1$ as an equation on $\sigma:=\rho-1$. Next, we  divide
 $\eqref{sys}_2$  by $\rho$ and 
 replace $\rho$ by $1+\sigma$ in the last term of $\eqref{sys}_2.$
 Taking advantage of \eqref{fraclap} and denoting  $\mu:=c_{d,\alpha}^{-1},$ 
  we eventually  obtain
\begin{align}\label{ren}
\left\{\begin{array}{rcl}
\sigma_t + \div ( \sigma u  )  &=& -\div u,\\[1.5ex] 
u_t + \mu  ( - \Delta)^{\alpha /2} u  &= &F(u,\sigma)
\end{array}
\right.
\end{align}
with 
\begin{multline}\label{iii}
\hspace{2cm} F(u, \sigma):= I_\alpha(u,\sigma) - \mu\sigma\,( - \Delta)^{\alpha /2} u -  (u \cdot\nabla) u\\
\andf I_\alpha (u , \sigma) :=  \pv \int_{\R^d} \biggl(\frac{u(x+y)- u(x)}{|y|^{d+\alpha}}\biggr) \bigl( \sigma (x+y) - \sigma (x) \bigr) \diff y.
\end{multline}
Throughout the paper we consider system $\eqref{ren}$, noting that it is equivalent to \eqref{sys} for sufficiently smooth solutions.
The main result of the paper restated in the above setting reads as follows:
\begin{theo}\label{main}
Assume that  $\alpha \in (1,2).$  There exists $\varepsilon>0$, such that the fractional Euler system $\eqref{ren}$ with initial data $u_0$ and $\sigma_0$  satisfying
\begin{equation}\label{eq:smallness1}
\|u_0\|_{\dot B^{2-\alpha}_{d,1}} + \|\sigma_0\|_{\dot B^{1}_{d,1}} <\varepsilon
\end{equation}
and $\nabla u_0\in \dot B^{2-\alpha}_{d,1},$ $\nabla\sigma_0\in\dot B^{1}_{d,1}$
admits a unique global in time solution $(\sigma,u)$ with
$$
u,\nabla u \in L^1(\R_+;\dot B^{2}_{d,1}) \cap \cC_b(\R_+; \dot{B}^{2-\alpha}_{d,1})
\andf \sigma,\nabla\sigma \in \cC_b(\R_+; \dot B^{1}_{d,1}).
$$
If Condition \eqref{eq:smallness1} if fulfilled only by $\sigma_0$ then there exists a time $T>0$ that may be bounded from below in terms of the norms of the data, and such that
system $\eqref{ren}$ with initial data $u_0$ and $\sigma_0$ admits a   unique  solution $(\sigma,u)$  on $[0,T]$ with
\begin{equation}\label{eq:ET}
u,\nabla u \in L^1(0,T;\dot B^{2}_{d,1}) \cap \cC_b([0,T]; \dot{B}^{2-\alpha}_{d,1})
\andf \sigma,\nabla\sigma \in \cC_b([0,T]; \dot B^{1}_{d,1}).\end{equation}
\end{theo}

Let us briefly present the main steps of the proof of Theorem \ref{main}.
Renormalization of \eqref{sys} into \eqref{ren} is the primary idea behind our approach, as it reduces our problem to the coupling of two  well known equations, namely
\begin{align}
\partial_t\sigma + \div(u\sigma) = f_1,\label{known1}\\
\partial_t u + \mu(-\Delta)^\frac{\alpha}{2}u = f_2.\label{known2}
\end{align}
As $\alpha\in(1,2)$ and the external forces $f_1$ and $f_2$ are smooth,  strong existence to equations \eqref{known1} and \eqref{known2} is relatively easy to obtain.
In our case however, the external forces are dependent on $u$ and $\sigma$, which makes the problem more complicated and pushes us to consider small initial data for getting a global statement. 
The main difficulty comes 
from the nonlocal bilinear term $I_\alpha(u,\sigma)$ appearing it \eqref{iii}. 
It will be handled thanks to  the following lemma:
\begin{lem}\label{cruc}
Assume that $\alpha\in(1,2)$ and that $\theta\in[0,\min(2-\alpha,\alpha-1)].$ 
 There exists $K>0$ such that for all  $u$ and $\sigma$
in $\cS(\R^d),$ we have 
\[ \| I_\alpha(u,\sigma) \|_{\dot B^{2- \alpha}_{d,1}} \leq K \|\nabla u\|_{ \dot B^{1-\theta}_{d,1}} \|\sigma\|_{\dot B^{1+\theta}_{d,1}} \ , \]
where $I_\alpha$ is defined by (\ref{iii}).
\end{lem} 

\medskip 

The proof of Theorem \ref{main} can be summarized by the following steps:

\begin{enumerate}
\item  We begin with the proof of Lemma \ref{cruc}.

\item  We introduce an iterative scheme that, somehow, decouples system \eqref{ren}. It is complemented with smooth and decaying initial data that approximate $(\sigma_0, u_0).$ 
\item The iterative scheme produces a sequence  $(\sigma^n,u^n)$ of global smooth approximate solutions that solve equations of the form \eqref{known1} and \eqref{known2}. 
Assuming that  \eqref{eq:smallness1} is fulfilled, we obtain 
uniform estimates for the sequence in the critical space  
$L^\infty\dot{B}^1_{d,1}\times\bigl(L^1\dot{B}^2_{d,1}\cap L^\infty\dot{B}^{2-\alpha}_{d,1}\bigr).$
\item  We estimate $(\sigma^n, u^n)$  in higher-order spaces by differentiating equations on $\sigma^n$ and $u^n$.
\item Thanks to the higher-order estimates we are able to establish that 
$(\sigma^n,u^n)$ is a Cauchy sequence in the space 
$$\cC([0,T];\dot B^1_{d,1})\times\bigl(\cC([0,T];\dot{B}^{2-\alpha}_{d,1})
\cap L^1_T\dot{B}^2_{d,1}\bigr),\quad\hbox{for all }\ T\geq0.$$
\item We show that the limit $(\sigma,u)$ of that Cauchy sequence fulfills system \eqref{ren}
and that it does have the regularity stated in Theorem \ref{main}. 
\item Assuming only that $\sigma_0$ is small, 
we prove the existence of a time $T>0$  and of a solution $(\sigma,u)$ to \eqref{ren}
in the space \eqref{eq:ET}. 
\item We give a  proof of  uniqueness that requires only $\sigma$  
to be small for one of the two solutions.
\end{enumerate}


\section{Proof of the main result}\label{prof}

We begin with the study of  the nonlocal term.

\subsection*{Step 1: Proof of Lemma \ref{cruc}}
For   $u$   in $\mathcal S(\R^d),$ one may write, using the notation of   \eqref{delta},
\begin{align*}
\di{y} u (x) = u(x+y) - u(x) = y \cdot \int_0^1 \nabla u(x+ ty) \diff t \ ,
\end{align*}
so that the operator $I_\alpha$ can be written as
\begin{multline*}
I_\alpha ( u , \sigma) (x)  = \int_{\R^d} \int_0^1 \Bigl(\frac y{|y|^{d+ \alpha}}\cdot  \nabla u (x +ty) \Bigr) \left( \sigma(x+y) - \sigma(x) \right) \diff t \diff y  \\
					 =  \underbrace{\int_{\R^d} \int_0^1 \Bigl( \frac{y}{|y|^{d+ \alpha}} \cdot \di{yt} \nabla u (x)\Bigr) \di{y} \sigma(x) \diff t \diff y}_{=:A(x)}  + 
					 \underbrace{\int_{ \R^d}  \Bigl(\frac{y}{|y|^{d+\alpha}} \cdot\nabla u (x)\Bigr) \di{y} \sigma (x) \diff y}_{=:B(x)} \, . 
\end{multline*}
In light of \eqref{nbes2} and triangular inequality, one thus has 
\begin{equation*}
\| I_\alpha  \|_{ \dot B^{2-\alpha}_{d,1}} \leq \|A \|_{ \dot B^{2-\alpha}_{d,1}}+\| B\|_{ \dot B^{2-\alpha}_{d,1}}
\leq C \left\| \frac{ \| \di{h} A ( \ \cdot \ ) \|_{d} }{|h|^{d+2-\alpha} } \right\|_{1} +\| B\|_{ \dot B^{2-\alpha}_{d,1}}.
\end{equation*} 
Now,  the trivial identity 
\begin{equation*}
\di{y} (gf) (x) = g(x+y) \di{y} f(x) + f(x) \di{y} g(x).
\end{equation*}
implies that \begin{equation*}
\di{h} \left( \di{ty} \nabla u (x) \cdot \di{y} \sigma (x) \right) = \di{ty} \nabla u(x+h) \cdot \di{h} \di{y} \sigma(x) + \di{y} \sigma (x) \cdot \di{h} \di{ty} \nabla u (x),
\end{equation*}
whence 
\begin{multline*}
\di{h} A(x) = \underbrace{\int_{\R^d} \int_0^1  \Bigl(\frac{y}{|y|^{d+ \alpha}} \cdot \di{ty} \nabla u(x+h)\Bigr)  \di{h} \di{y} \sigma(x) \diff t \diff y}_{=:E(x,h)} + \\ + 
\underbrace{\int_{\R^d} \int_0^1  \Bigl(\frac{y}{|y|^{d+ \alpha}} \di{y} \sigma (x)\Bigr) \cdot \di{h} \di{ty} \nabla u (x) \diff t \diff y}_{=:F(x,h)}
\end{multline*}
which leads to
\begin{equation*}
 \left\| \frac{\|\di{h}  A ( \ \cdot  \ ) \|_d}{|h|^{d+2-\alpha}} \right\|_1 \leq \left\| \frac{\|E( \ \cdot  \ , h ) \|_d}{|h|^{d+2-\alpha}} \right\|_1 + 
 \left\| \frac{\|F ( \ \cdot  \ , h ) \|_d}{|h|^{d+2-\alpha}} \right\|_1 \ .
\end{equation*}
We estimate the term $ \left\| \frac{\|E ( \ \cdot  \ , h ) \|_d}{|h|^{d+2-\alpha}} \right\|_1$ in the following way (where exponents $q$ and $q'$ fulfill the relation $1/q+1/q'=1/d$):
$$
\begin{aligned}
 		 \left\| \frac{\|E ( \ \cdot  \ , h ) \|_d}{|h|^{d+2-\alpha}} \right\|_1 &= \\
= 		& \int_{\R^d} \frac{1}{|h|^{d+2-\alpha}}\left( \int_{\R^d} \left| \int_{\R^d} \int_0^1 \!\frac{ y }{|y|^{d+ \alpha}} \cdot \di{ty} \nabla u (x+h) \cdot \di{h} \di{y} \sigma(x) \diff t \diff y \right|^d \!\!\diff x \right)^{1/d}\!\! \diff h \\
\leq 	& \int_{\R^d} \frac{1}{|h|^{d+2-\alpha}}  \int_{\R^d}  \frac{ 1}{|y|^{d+ \alpha-1}} \cdot  \int_0^1\left\|  \di{ty} 
\nabla u ( \ \cdot \ + h) \cdot \di{h} \di{y} \sigma( \ \cdot \ )  \right\|_d \diff t  \diff y  \diff h \\
\leq 	& \int_{\R^d} \frac{1}{|h|^{d+2-\alpha}}  \int_{\R^d}  \frac{ 1}{|y|^{d+ \alpha-1}} \cdot
 \int_0^1\left\|  \di{ty} \nabla u ( \ \cdot \ + h) \right\|_q  \cdot \left\| \di{h} \di{y} \sigma  \right\|_{q'} \diff t  \diff y  \diff h \\
\leq	& \int_{\R^d} \frac{1}{|h|^{d+2-\alpha}}   \int_{\R^d}  \frac{ 1 }{|y|^{d+ \alpha - 1}}  \int_0^1\left\|  \di{ty} \nabla u\right\|_q  \cdot \left\| \di{h} \di{y} \sigma  \right\|_{q'} \diff t  \diff y  \diff h \\
=		&  \int_{\R^d}  \int_{\R^d}  \frac{\left\| \di{h} \di{y} \sigma  \right\|_{q'}}{|h|^{d+2-\alpha}} \cdot   \int_0^1 \frac{ \left\|  \di{ty} \nabla u\right\|_q }{|y|^{d+ \alpha -1}}   \diff t  \diff y  \diff h \\
\leq	&   \int_{\R^d} \sup_y \left\{    \frac{\left\| \di{h} \di{y} \sigma \right\|_{q'}}{|h|^{d+2-\alpha}} \right\} \cdot  
\int_0^1\int_{\R^d}  \frac{ \left\|  \di{ty} \nabla u\right\|_q }{|y|^{d+ \alpha -1}}   \diff t  \diff y  \diff h  \\
= 		& \frac 1\alpha  \int_{\R^d}  \frac{ \left\|  \di{z} \nabla u\right\|_q }{|z|^{d+ \alpha -1}} \diff z \cdot \int_{\R^d} \sup_y \left\{    \frac{\left\| \di{h} \di{y} \sigma  \right\|_{q'}}{|h|^{d+2-\alpha}} \right\} \diff h \ .
\end{aligned}
$$
Here we have used the fact that the norm of a function is translation invariant. As $\| \di{h} \di{y} \sigma \|_{q'} = \| \di{y} \di{h} \sigma \|_{q'} \leq 2 \| \di{h} \sigma \|_{q'}$ 
we finally obtain
\begin{equation*}
\left\| \frac{\|E ( \ \cdot  \ , h ) \|_d}{|h|^{d+2-\alpha}} \right\|_1 \leq C  \int_{\R^d}  \frac{ \left\|  \di{z} \nabla u\right\|_q }{|z|^{d+ \alpha -1}} \diff z \cdot \int_{\R^d} \frac{\left\|  \di{h} \sigma  \right\|_{q'}}{|h|^{d+2-\alpha}} \diff h \ .
\end{equation*}
We use Lemma \ref{super} to determine what the exponent $q$  should be  
to match with  our assumptions that $\nabla u \in \dot B^{1-\theta}_{d,1}$ and $\sigma \in \dot B^{1+\theta}_{d,1}$. 
\begin{itemize}
\item In order to obtain the embedding $ \dot B^{1-\theta}_{d,1} \subset \dot B^{\alpha - 1}_{q,1},$ [$2-\alpha \geq \theta$] we need 
$q\geq d$ and
\begin{equation*}
\frac {1+\theta}d = \frac {2- \alpha}d + \frac 1q \cdotp
\end{equation*}
\item In order to obtain the embedding $ \dot B^{1+\theta}_{d,1} \subset \dot B^{2-\alpha}_{q',1},$ we need 
$q'\geq d$ and
\begin{equation*}
\frac {1-\theta}d = \frac {\alpha-1}d + \frac 1{q'} \cdotp
\end{equation*}
\end{itemize}
These two conditions uniquely determine the exponents to be
\begin{equation*}
q = \frac{d}{\alpha+\theta -1} \andf q'= \frac{d}{2- \alpha-\theta}\cdotp
\end{equation*}
Clearly, $1/q+1/q'=1/d$ and our conditions on $\theta$ guarantee that  $q,q'\geq d.$
So, using the above embeddings, we conclude that
$$
\left\| \frac{\|E ( \ \cdot  \ , h ) \|_d}{|h|^{d+2-\alpha}} \right\|_1  \leq C \ \| \nabla u\|_{\dot B^{\alpha -1}_{q,1}} \| \sigma \|_{\dot B^{2-\alpha}_{q',1}} \leq C \ \| \nabla u \|_{\dot B^{1-\theta}_{d,1}} \| \sigma \|_{\dot B^{1+\theta}_{d,1}}.
$$
Analogous calculations can be performed for bounding the term with $F$: we introduce another two exponents $r$ and $r'$
such that $1/r+1/r'=1/d,$ and write that
$$
\begin{aligned}
 \left\| \frac{\|F( \ \cdot  \ , h ) \|_d}{|h|^{d+2-\alpha}} \right\|_1 & = \int_{\R^d} \frac{1}{|h|^{d+2-\alpha}}\left( \int_{\R^d} \left| \int_{\R^d} \int_0^1 \Bigl(\frac{ y}{|y|^{d+ \alpha}}  \di{y} \sigma (x)\Bigr) \cdot \di{h} \di{ty} \nabla u (x) \diff t \diff y \right|^d \!\diff x \right)^{1/d}\!\! \diff h \\ 
 & \leq \int_{\R^d} \frac 1{|h|^{d+2 - \alpha}} \int_{\R^d} \frac {1}{|y|^{d+\alpha-1}} \int_0^1 \| \di{y} \sigma ( \ \cdot \ ) \di{h} \di{ty} \nabla u ( \ \cdot \ ) \|_d \diff t \diff y \diff h \\
 & \leq \int_{\R^d} \frac 1{|h|^{d+2 - \alpha}} \int_{\R^d} \frac 1{|y|^{d+\alpha-1}} \int_0^1 \| \di{y} \sigma ( \ \cdot \ )\|_r  \| \di{h} \di{ty} \nabla u ( \ \cdot \ ) \|_{r'} \diff t \diff y \diff h \\
 & \leq \int_{\R^d} \frac 1{|h|^{d+2-\alpha}} \int_{\R^d} \frac 2{|y|^{d+\alpha-1}}  \| \di{y} \sigma ( \ \cdot \ )\|_r  \| \di{h}  \nabla u( \ \cdot \ ) \|_{r'} \diff y \diff h \\
 & \leq C \| \sigma \|_{\dot B^{\alpha-1}_{r,1}} \| \nabla u \|_{\dot B^{2 - \alpha}_{r',1}}.
\end{aligned}
$$
At this stage, we use the following embeddings (keeping in mind that $0\leq \theta\leq\alpha-1$):
$$\begin{aligned}
\dot B^{1+\theta}_{d,1}\hookrightarrow \dot B^{\alpha-1}_{r,1}&\with 
r=\frac{d}{\alpha-\theta-1}\\
\andf 
\dot B^{1-\theta}_{d,1}\hookrightarrow \dot B^{2-\alpha}_{r',1}&\with 
r'=\frac d{2-\alpha+\theta},\end{aligned}
$$
and eventually  get 
$$
 \left\| \frac{\|F( \ \cdot  \ , h ) \|_d}{|h|^{d+2-\alpha}} \right\|_1\leq  C \ \| \sigma\|_{\dot B^{1+\theta}_{d,1}} \| \nabla u \|_{\dot B^{1-\theta}_{d,1}}.
$$
For the  term $B,$ we  just have to use the fact  that 
\begin{equation*}
B = \nabla u \cdot T\sigma\with T\sigma(x):= 
\int_{\R^d} \frac{y}{|y|^{d + \alpha}}(\sigma(x+y) - \sigma (x) ) \diff y.
\end{equation*}
The Fourier multiplier  corresponding to $T$ may be computed as follows:
\begin{align*}
\mathcal F \{T \sigma \} ( \xi )  & = \int_{\R^d} \frac{y}{|y|^{d+ \alpha}} \widehat \sigma ( \xi) \left( e^{i \langle y , \xi \rangle} - 1 \right) \diff y \\
 & = \widehat \sigma ( \xi) | \xi |^{\alpha -1} \underbrace{\int_{\R^d} \frac{z}{|z|^{d  + \alpha}} \left( e^{i\langle z , \frac{\xi}{| \xi |} \rangle} - 1 \right) \diff z }_{=:R(\xi)}\ .
\end{align*}
The function $R $ is  smooth away of zero, constant on rays and bounded. 
Furthermore,  that $d+\alpha > d+1$  guarantees convergence at infinity, while  the features of 
$\Bigl( e^{i\langle z , \frac{\xi}{| \xi |} \rangle} - 1 \Bigr)$ controls the convergence near $z=0,$ since $\alpha<2.$

In order to show that the operator $T$ is bounded from 
$\dot B^{1}_{d,1}$ to $\dot B^{2 -\alpha}_{d,1},$ one may proceed  as follows. 
Take some smooth function  $\psi$ compactly supported away from $0$ and with value $1$
 in the neighborhood of $\supp \varphi$ (the function used in the definition of dyadic blocks). Then we have
\begin{equation*}
\mathcal F \{ \dot \Delta_j T \sigma \} ( \xi) = \varphi( 2^{-j} \xi ) | \xi |^{\alpha-1} R( \xi )  \widehat \sigma (\xi) = \left( \psi(2^{-j} \xi ) | \xi |^{\alpha-1} R ( \xi) \right) \mathcal F \{ \dot \Delta_j \sigma \} (\xi)
\end{equation*}
and by defining
\begin{equation*}
g_j(x) = \frac1{(2\pi)^d}\int_{\R^d} \psi( 2^{-j} \xi ) | \xi |^{\alpha -1 }R(\xi) e^{i \langle x, \xi \rangle} \diff \xi
\end{equation*}
we obtain by Young's inequality for convolutions:
\begin{equation*}
\| \dot \Delta_j T\sigma \|_d = \| g_j  * \dot \Delta_j \sigma \|_d \leq \| g_j \|_1 \| \dot \Delta_j \sigma \|_d \ .
\end{equation*}
Since, by a change of variables, we have
\begin{equation*}
g_j (x) = 2^{j(\alpha -1)} 2^{jd} \int_{\R^d} \psi(\eta) | \eta|^{\alpha-1} R(\eta) e^{i \langle 2^{-j} x , \eta \rangle } \diff \eta,
\end{equation*}
it is sufficient to show that the $L^1$-norm of the function
\begin{equation*}
g(x) = \int_{\R^d} \psi(\eta) |\eta|^{\alpha-1} R(\eta) e^{i \langle x , \eta\rangle } \diff \eta
\end{equation*}
is finite because $\| g_j ( \cdot ) \|_1 = 2^{j(\alpha-1)} \| 2^{jd} g(2^j \cdot ) \|_1 = \| g \|_1$. Using integration by parts we write
\begin{equation*}
\begin{split}
g(x) 	& = (1 +|x|^2)^{-d } \int_{\R^d} \psi(\eta) |\eta|^{\alpha-1} R(\eta) (\operatorname{Id} - \Delta_\eta)^d \left( e^{i \langle x , \eta \rangle } \right) \diff\eta \\
		& = (1 +|x|^2)^{-d } \int_{\R^d} (\operatorname{Id}- \Delta_\eta)^d \left( \psi(\eta) |\eta|^{\alpha-1} R(\eta)\right)  e^{i \langle x , \eta \rangle } \diff \eta\end{split}
\end{equation*}
from which it follows that $g$ is in $L^1$, since $| g(x) | \leq C (1+ |x|^2)^{-d}$. Hence
\begin{equation*}
\| \dot \Delta_j T \sigma \|_d \leq C 2^{j(\alpha -1)} \| \dot \Delta_j \sigma \|_d\quad\hbox{for all }\ j\in\Z.
\end{equation*}
By the assumption that $\sigma \in \dot B^{1+\theta}_{d,1},$ this  implies that $T\sigma \in \dot B^{2+\theta- \alpha}_{d,1}$.
Finally, since  the product maps $\dot B^{1-\theta}_{d,1} \times \dot B^{2+\theta-\alpha}_{d,1}$ to $\dot B^{2-\alpha}_{d,1}$ by (h) in Lemma \ref{super} if $0\leq \theta\leq \alpha-1,$ 
we get the desired estimate for $B.$\qed

\subsection*{Step 2: Iterative scheme}

Fix some nonnegative integer $n_0$ and define  
$$u_0^n:=\sum_{|j|\leq n+n_0}\dot\Delta_j u_0\andf \sigma_0^n:=\sum_{|j|\leq n+n_0}\dot\Delta_j \sigma_0.$$
It is clear that those two 
sequences belongs to the space $W^\infty_d$ of smooth functions with 
derivatives of any order in $L^d,$ and that 
\begin{equation*}
\begin{split}
&u^n_0\to u_0\andf \nabla u^n_0\to\nabla u_0 \quad \mbox{in}\ \dot{B}^{2-\alpha}_{d,1},\\
&\sigma^n_0\to\sigma_0\andf \nabla\sigma^n_0\to\nabla\sigma_0\quad\mbox{in}\ \dot{B}^1_{d,1}.
\end{split}
\end{equation*}
In what follows, we set\footnote{By taking $n_0$ large enough, 
  one may make $\varepsilon$ and $\eta$
as close as $\|u_0\|_{\dot B^{2-\alpha}_{d,1}}$ and $\|a_0\|_{\dot B^1_{d,1}}$ 
as we want.} 
\begin{equation}\label{eq:defsmall}
\varepsilon:=\sup_{n\in\N} \|u^n_0\|_{\dot B^{2-\alpha}_{d,1}}
\andf \eta:=\sup_{n\in\N} \|a^n_0\|_{\dot B^{1}_{d,1}}.\end{equation}
We then introduce the following iterative scheme: for $n=0,$ let
\begin{align*}
\sigma^0\equiv \sigma^0_0,\quad u^0\equiv 0,
\end{align*}
while for all $n=1,2,...$, let $u^n$ be the solution to the system
\begin{align}\label{un}
u^n_t + \mu ( - \Delta)^{\alpha /2} u^n  = F(u^{n-1},\sigma^{n-1})
\end{align}
with the initial data $u^n(0)=u^n_0$ and $F$ defined in \eqref{iii}. Finally, let $\sigma^n$ be a solution to 
\begin{align}\label{rhon}
\sigma^n_t + \div ( \sigma^n u^n  ) & = -\div u^n
\end{align}
with the initial data $\sigma^n(0) = \sigma^n_0$.

\subsection*{Step 3: Existence of approximate solutions and estimates for critical regularity}

Given that the data are  in $W^\infty_d,$ 
the existence of a sequence  $(\sigma^n,u^n)_{n\in\N}$ of global approximate solutions in $\cC^1(\R_+;W^\infty_d)$   follows from an easy induction argument, based on the two propositions below, that are proved  in appendix. 
\begin{prop}\label{rubel}
Let $u:\R_+\times\R^d\to\R^d$ be a given time-dependent vector field  with components in $\cC(\R_+;W^\infty_d).$ 
Then  the transport equation
\begin{align}\label{trans}
\sigma_t   + {\rm div} \left( u \sigma  \right) = f 
\end{align}
with initial data $\sigma_0$ in $W^\infty_d$ and source term $f$ in $\cC(\R_+;W^\infty_d)$ admits a 
unique solution $\sigma$ in $\cC^1(\R_+;W^\infty_d)$.
Moreover, there exists a constant $C>0$, such that for all $T\geq0,$ 
 \begin{align*}
\|\sigma\|_{L^\infty_T\dot{B}^1_{d,1}}\leq 
\|\sigma_0\|_{\dot{B}^1_{d,1}} + \|f\|_{L^1_T\dot{B}^1_{d,1}}+C\int_0^T\|\nabla u(t)\|_{\dot{B}^{1}_{d,1}}
\|\sigma(t)\|_{\dot{B}^{1}_{d,1}}\diff t.
\end{align*}
\end{prop}
Taking  $\sigma=\sigma^n$, $u=u^n$ and $f=-\div u^n,$ and using Lemma \ref{super}\ref{besovineq}
  and Gronwall lemma, we readily get:
\begin{prop}\label{nazwa}
For all $n=1,2,...$, equation \eqref{rhon} admits a unique global solution $\sigma^n\in \cC^1(\R_+;W^\infty_d)$
and there exists a constant $C_1$  such that  for all $T>0,$ 
 $$
\|\sigma^n\|_{L^\infty_T\dot{B}^1_{d,1}}\leq  e^{C_1\|u^n\|_{L^1_T\dot{B}^{2}_{d,1}}}\left(\|\sigma^n_0\|_{\dot{B}^1_{d,1}} + \|u^n\|_{L^1_T\dot{B}^2_{d,1}}\right)\cdotp $$\end{prop}

Secondly, we prove existence and present an estimate for $u^n$ depending on $u^{n-1}$ and $\sigma^{n-1}$. Here, as in the density case, we use a classical result, the  proof of which may  be found in the appendix.

\begin{prop} \label{fracdiffestimates}
Let $0<T \leq \infty$, $1 \leq p \leq \infty$ and $s \in \R$. Assume that $u_0 \in \dot B^s_{p,1}$ and $f \in L^1_T \dot B^s_{p,1}$. Then the Cauchy problem
\begin{equation}\label{fracdiff}
\left\{
\begin{array}{lll}
u_t + \mu (-\Delta)^{\alpha /2 } u = f(t,x) , & (t,x) \in \R^+ \times \R^d \\
u(0,x) = u_0 (x) ,  & x \in \R^d
\end{array}
\right.
\end{equation}
 has a unique solution $u \in \cC([0,T); \dot B^s_{p,1}) \cap L^1_T \dot B^{s+ \alpha}_{p,1}$ and there exists $C>0$ such that
\begin{equation*}
\| u \|_{ L^1_T \dot{B}^{s+\alpha}_{p,1}} +  \| u \|_{ L^\infty_T \dot B^s_{p,1}} \leq C \left[  \| u_0 \|_{\dot B^s_{p,1}} + \| f \|_{L^1_T \dot B^s_{p,1}} \right] \cdotp
\end{equation*}
\end{prop}
The above proposition  enables us to  get the following result for  $u^n$. 
\begin{prop}
For all $n=1,2,...$, given $u^{n-1}$ and $\sigma^{n-1}$ in  $\cC(\R_+;W^\infty_d)$, there exists a unique solution 
$u^n\in \cC^1(\R_+; W^\infty_d)$ to \eqref{un} such that for all $T\geq0,$ 
\begin{equation}\label{ul1}
\|u^n\|_{L^1_T \dot B^{2}_{d,1}} + \| u^n \|_{L^\infty_T \dot  B^{2-\alpha}_{d,1} } 
\leq  \|u^n_0\|_{\dot{B}^{2-\alpha}_{d,1}}+C_2\int_0^T   \| u^{n-1} \|_{\dot B^2_{d,1}}\bigl( \| \sigma^{n-1} \|_{\dot B^1_{d,1}}+ \| u^{n-1} \|_{\dot B^{2-\alpha}_{d,1}}\bigr)\,dt.
\end{equation}
\end{prop}

\begin{proof}
The proof is an application of Proposition \ref{fracdiffestimates} with $s = 2- \alpha$, $p=d$, and 
\begin{equation*}
f = F(u^{n-1},\sigma^{n-1}).
\end{equation*}
We estimate the norm $\|f \|_{L^1_T \dot B^{2 - \alpha}_{d,1}}$ dealing with each term separately. Bounding  the first term follows directly from Lemma \ref{cruc} and H\"older's inequality:
$$
\| I_\alpha(u^{n-1},\sigma^{n-1})\|_{L^1_T \dot B^{2 - \alpha}_{d,1}}	\leq   C \int_0^T  \| \nabla u^{n-1} \|_{\dot B^1_{d,1}}  \|  \sigma^{n-1} \|_{\dot B^1_{d,1}}\diff t.$$
Lemma \ref{super} and H\"older's inequality allow us to estimate norms of the second and the third terms of $f$ 
as follows:
\begin{equation*}
\begin{split}
\|\sigma^{n-1} ( - \Delta)^{\alpha /2} u^{n-1} \|_{L^1_T \dot B^{2 - \alpha}_{d,1}} 	&  \leq C  \int_0^T \| \sigma^{n-1} \|_{\dot B^1_{d,1}}  \|( - \Delta)^{\alpha /2} u^{n-1} \|_{\dot B^{2 - \alpha}_{d,1}} \diff t \\
 & \leq C \int_0^T \| \sigma^{n-1} \|_{\dot B^1_{d,1} } \| u^{n-1} \|_{\dot B^{2}_{d,1}} \diff t \end{split}
\end{equation*}
and
\begin{equation*}
\begin{split}
\| (u^{n-1} \cdot\nabla) u^{n-1} \|_{L^1_T \dot B^{2 - \alpha}_{d,1}}	& \leq C  \int_0^T \| u^{n-1} \|_{\dot B^{2-\alpha}_{d,1}}  \| \nabla u^{n-1} \|_{\dot B^1_{d,1}} \diff t \\
& \leq  C \int_0^T \| u^{n-1} \|_{\dot B^{2-\alpha}_{d,1}}  \| u^{n-1} \|_{\dot B^2_{d,1}} \diff t.
 	\end{split}
\end{equation*}
Thus we have the inequality
\begin{equation*}
\|f \|_{L^1_T \dot B^{2 - \alpha}_{d,1}} \leq  C\int_0^T   \| u^{n-1} \|_{\dot B^2_{d,1}}\bigl( \| \sigma^{n-1} \|_{\dot B^1_{d,1}}+ \| u^{n-1} \|_{\dot B^{2-\alpha}_{d,1}}\bigr)\diff t,\end{equation*}
from which we get the stated result.
\end{proof}
It is now easy to check  that if $\varepsilon$ and $\eta$ in \eqref{eq:defsmall} have been 
taken sufficiently small then we have for all $n\in\N$ and $T\geq0,$
\begin{equation}\label{eq:smallcritic}
\|u^n\|_{L^1_T B^{2}_{d,1}} + \| u^n\|_{L^\infty_T \dot  B^{2-\alpha}_{d,1} }\leq 2\varepsilon
\andf \|\sigma^n\|_{L^\infty_TB^1_{d,1}}\leq 2\eta.\end{equation} 
Indeed, the result is obviously true for $n=0,$ and if it is true for $n-1$ 
then Inequality \eqref{ul1} implies that 
$$
\|u^n\|_{L^1_T B^{2}_{d,1}} + \| u^n\|_{L^\infty_T \dot  B^{2-\alpha}_{d,1} }\leq \varepsilon
+4C_2\varepsilon(\varepsilon+\eta).
$$
Hence \eqref{eq:smallcritic} is fulfilled by $u^n$ if $\varepsilon$ and $\eta$ are so small that
\begin{equation*}
4C_2(\varepsilon+\eta)\leq1.
\end{equation*}
Then, taking advantage of Proposition \ref{nazwa}, we get
$$
 \|\sigma^n\|_{L^\infty_TB^1_{d,1}}\leq  e^{2C_1\varepsilon} (\eta+2C_1\varepsilon)\leq 2\eta
 $$
 provided that in addition we have for instance
 \begin{equation}\label{eq:eps2} 2C_1\varepsilon \leq \log(3/2)\andf 6C_1\varepsilon\leq \eta.\end{equation}


\subsection*{Step 4: Higher-order estimates}

In order to obtain higher-order estimates we simply differentiate equations \eqref{un} and \eqref{rhon} with respect to the space variable. 
Denoting  by $\partial_k$ the derivative with respect to the variable $x_k,$  we get for $k=1,\cdots,d,$ 
$$
(\partial_k\sigma^n)_{t} + {\rm div}(u^n\partial_k \sigma^n) = -\div(\partial_k u^n\,\sigma^n) - \div\partial_k u^n\andf
$$
 \begin{multline} \label{itermomentumdiff}
(\partial_k u^n)_{t} +  \mu (- \Delta)^{\alpha/2}\partial_k  u^n   =  -\mu \partial_k \sigma^{n-1} ( - \Delta)^{\alpha /2} u^{n-1} -\mu \sigma^{n-1} (- \Delta)^{\alpha/2} 
\partial_k u^{n-1}\\ + I_\alpha (\partial_k u^{n-1}, \sigma^{n-1}) + I_\alpha (u^{n-1},\partial_k \sigma^{n-1}) - (\partial_k u^{n-1} \cdot \nabla) u^{n-1}  - ( u^{n-1} \cdot \nabla) \partial_k u^{n-1}\, .
\end{multline}

\begin{prop}\label{nazwax}
Approximate solutions $\sigma^n$ to \eqref{rhon} satisfy
$$\|\nabla \sigma^n\|_{L_T^\infty\dot{B}^1_{d,1}}\leq e^{C_3\|\nabla u^n\|_{L^1_T\dot{B}^1_{d,1}}}\Big[\|\nabla\sigma^n_0\|_{\dot{B}^1_{d,1}} + C_3\|\nabla u^n\|_{L^1_T\dot{B}^2_{d,1}} \bigl(\|\sigma^n\|_{L^\infty_T\dot{B}^1_{d,1}}+1\bigr)\Big]\cdotp
$$
\end{prop}

\begin{proof}
We apply Proposition \ref{rubel} with $\sigma=\partial_k \sigma^n$, $u=u^n$ and $f=-\div(\partial_ku^n\,\sigma^n) - \div\partial_k u^n$. We need to estimate $f$. Clearly, we have
\begin{align*}
\|f\|_{L^1_T\dot{B}^1_{d,1}}\leq  \|\sigma^n\nabla\div u^n\|_{L^1_T\dot{B}^1_{d,1}} + \|\nabla u^n\otimes \nabla\sigma^n\|_{L^1_T\dot{B}^1_{d,1}} 
+ \|\nabla \div u^n\|_{L^1_T\dot{B}^1_{d,1}}
\end{align*}
and with the use of H\" older's inequality and Lemma \ref{super}, we get
$$
\|f\|_{L^1_T\dot{B}^1_{d,1}}\leq C_3
\biggl(\|\nabla u^n\|_{L^1_T\dot{B}^2_{d,1}}\left(\|\sigma^n\|_{L^\infty_T\dot{B}^1_{d,1}}+1\right) 
+ \int_0^T \|\nabla u^n\|_{\dot{B}^1_{d,1}}\|\nabla\sigma^n\|_{\dot{B}^1_{d,1}}\,dt\biggr)\cdotp
$$
By Proposition \ref{rubel} and Gronwall lemma, this  leads to the desired inequality.
\end{proof}

Likewise, applying Proposition \ref{fracdiffestimates} to equation \eqref{itermomentumdiff} provides us with  higher-order bounds for~$u^n$:

\begin{prop} \label{itermomentumdiffestimate}
Approximate solutions $u^n$ to \eqref{un} satisfy for all $T\geq0,$
$$\displaylines{
\|\nabla u^n\|_{L^1_T \dot B^{2}_{d,1}\cap L^\infty_T \dot  B^{2-\alpha}_{d,1} } \leq C_4  \Big[ \|\nabla u^n_0\|_{\dot{B}^{2-\alpha}_{d,1}} + \|\nabla u^{n-1}\|_{L^1_T \dot B^2_{d,1}}\| \sigma^{n-1}\|_{L^\infty_T \dot B^1_{d,1}}\hfill\cr\hfill 
+ \| u^{n-1}\|_{L^1_T \dot B^2_{d,1}} \|\nabla \sigma^{n-1} \|_{L^\infty_T \dot B^1_{d,1}}
+   \|\nabla  u^{n-1}\|_{L^1_T \dot B^2_{d,1}}\| u^{n-1}\|_{L^\infty_T \dot B^{2 -\alpha}_{d,1}} +  \| u^{n-1}\|_{L^1_T \dot B^2_{d,1}}\| \nabla u^{n-1} \|_{L^\infty_T \dot B^{2 -\alpha}_{d,1}}   \Big]\cdotp}$$
\end{prop}

\begin{proof}
We  now apply Proposition \ref{fracdiffestimates} to $\partial_k u$ in equation \eqref{itermomentumdiff}, with $f$ being the right-hand side,  $s = 2- \alpha$
and $p=d$. Below, we bound  the $L^1_T \dot B^{2-\alpha}_{d,1}$ norm of each term of $f$ separately, using repeatedly 
Lemmas \ref{super} and \ref{cruc} and H\"older's inequality:
\begin{equation*}
\begin{split}
 \|\sigma^{n-1} (-\Delta)^{\alpha/2} \partial_ku^{n-1}\|_{L^1_T \dot B^{2-\alpha}_{d,1}} & \leq C \int_0^T \| \sigma^{n-1} \|_{\dot B^1_{d,1}} \| ( - \Delta)^{\alpha /2}  \partial_k u^{n-1}\|_{\dot B^{2-\alpha}_{d,1}}  \diff t  \\
 & \leq C \int_0^T \| \sigma^{n-1} \|_{\dot B^1_{d,1}} \|   \partial_ku \|_{\dot B^{2}_{d,1}}  \diff t \\
 & \leq C \| \partial_k u^{n-1}\|_{L^1_T \dot B^2_{d,1}}\| \sigma^{n-1} \|_{L^\infty_T \dot B^1_{d,1}};
\end{split}
\end{equation*}
\begin{equation*}
\begin{split}
 \|  \partial_k\sigma^{n-1} (-\Delta)^{\alpha/2} u^{n-1} \|_{L^1_T \dot B^{2-\alpha}_{d,1}} & \leq C \int_0^T \|  \partial_k\sigma^{n-1} \|_{\dot B^1_{d,1}} \| ( - \Delta)^{\alpha /2} u^{n-1} \|_{\dot B^{2-\alpha}_{d,1}}  \diff t \\ 
 & \leq C \| u^{n-1} \|_{L^1_T \dot B^2_{d,1}}\|  \partial_k\sigma^{n-1} \|_{L^\infty_T \dot B^1_{d,1} };
\end{split}
\end{equation*}
\begin{equation*}
\begin{split}
  \|I_\alpha[  \partial_ku^{n-1}, \sigma^{n-1}] \|_{L^1_T \dot B^{2-\alpha}_{d,1}} & \leq C \int_0^T \|  \partial_ku^{n-1}\|_{\dot B^2_{d,1}} \|  \sigma^{n-1} \|_{\dot B^{1}_{d,1}}  \diff t \\  & \leq C   \|  \partial_ku^{n-1} \|_{L^1_T \dot B^2_{d,1}}\| \sigma^{n-1} \|_{L^\infty_T \dot B^1_{d,1}};
\end{split}
\end{equation*}
\begin{equation*}
\begin{split}
  \|I_\alpha[ u^{n-1},  \partial_k\sigma^{n-1}] \|_{L^1_T \dot B^{2-\alpha}_{d,1}} 
  & \leq C \int_0^T \| u^{n-1} \|_{\dot B^2_{d,1}} \|  \partial_k\sigma^{n-1}\|_{\dot B^{1}_{d,1}}  \diff t  \\
  & \leq C \| u^{n-1}\|_{L^1_T \dot B^2_{d,1}} \| \partial_k\sigma^{n-1} \|_{L^\infty_T \dot B^1_{d,1}};
\end{split}
\end{equation*}
\begin{equation*}
\begin{split}
  \| u^{n-1} \cdot   \partial_k\nabla u^{n-1} \|_{L^1_T \dot B^{2-\alpha}_{d,1}} & \leq C \int_0^T \| u^{n-1} \|_{\dot B^{2- \alpha}_{d,1}} \| \partial_k\nabla u^{n-1} \|_{\dot B^{1}_{d,1}}  \diff t   \\
  & \leq C  \|  \partial_ku^{n-1}\|_{L^1_T \dot B^2_{d,1}}\|u^{n-1}\|_{L^\infty_T \dot B^{2 -\alpha}_{d,1}};
\end{split}
\end{equation*}
\begin{equation*}
\begin{split}
 \|  \partial_ku^{n-1} \cdot \nabla u^{n-1} \|_{L^1_T \dot B^{2-\alpha}_{d,1}} & \leq C \int_0^T \|  \partial_ku^{n-1} \|_{\dot B^{2- \alpha}_{d,1}} \| \nabla u^{n-1} \|_{\dot B^{1}_{d,1}}  \diff t   \\ & \leq C  \|u^{n-1}\|_{L^1_T \dot B^2_{d,1}}\|  \partial_ku^{n-1} \|_{L^\infty_T \dot B^{2 -\alpha}_{d,1}}.
\end{split}
\end{equation*}
 Putting together all these estimates and summing on $k=1,\cdots,d$  completes the proof.
\end{proof}

Denote $S_0:=\sup_{n\in\N} \|\nabla\sigma_0^n\|_{\dot B^1_{d,1}}\andf
U'_0:=\sup_{n\in\N} \|\nabla u_0^n\|_{\dot B^{2-\alpha}_{d,1}}.$
{}From Proposition \ref{nazwax} and \eqref{eq:smallcritic}, we get for all $T\geq0,$
taking $C_3$ slightly larger if need be,
\begin{equation}\label{eq:Dsn}
\|\nabla\sigma^n\|_{L^\infty_T\dot B^1_{d,1}}\leq C_3\bigl(S_0+\|\nabla u^n\|_{L_T^1\dot B^2_{2,1}}\bigr)
\end{equation}
and, thanks to Proposition \ref{itermomentumdiffestimate}, keeping in mind \eqref{eq:eps2}, 
\begin{equation}\label{eq:Dun1}
\|\nabla u^n\|_{L^\infty_T\dot B^{2-\alpha}_{d,1}\cap L^1_T\dot B^2_{d,1}}\leq C_4\bigl(U'_0
+\eta\|\nabla u^{n-1}\|_{L^\infty_T\dot B^{2-\alpha}_{d,1}\cap L_T^1\dot B^2_{2,1}}
+\varepsilon\|\nabla\sigma^{n-1}\|_{L^\infty_T\dot B^1_{d,1}}\bigr)\cdotp
\end{equation}
Therefore, combining \eqref{eq:Dsn} and \eqref{eq:Dun1}, we find that 
$$\displaylines{
\|\nabla u^n\|_{L^\infty_T\dot B^{2-\alpha}_{d,1}\cap L^1_T\dot B^2_{d,1}}+\frac1{2C_3}
\|\nabla\sigma^n\|_{L^\infty_T\dot B^1_{d,1}}\leq \frac{S_0}2+C_4U'_0 \hfill\cr\hfill+\frac12
\|\nabla u^n\|_{L_T^1\dot B^2_{2,1}} +\eta C_4\|\nabla u^{n-1}\|_{L^\infty_T\dot B^{2-\alpha}_{d,1}\cap L_T^1\dot B^2_{2,1}}+\varepsilon C_4\|\nabla\sigma^{n-1}\|_{L^\infty_T\dot B^1_{d,1}}.}$$
Provided $\eta$ and $\varepsilon$ also fulfill
$$
2\eta C_4\leq 1/2\andf 2\varepsilon C_3 C_4\leq 1/2,$$
one can then get by induction  the following uniform bound for all $T\geq0$:
\begin{equation}\label{eq:boundhigh}
\|\nabla u^n\|_{L^\infty_T\dot B^{2-\alpha}_{d,1}\cap L^1_T\dot B^2_{d,1}}+\frac1{C_3}
\|\nabla\sigma^n\|_{L^\infty_T\dot B^1_{d,1}}\leq  2S_0 +4C_4U'_0.
\end{equation}


\subsection*{Step 5: Convergence estimates}
Previous steps  established that the sequence of approximate solutions $(\sigma^n,u^n)_{n\in\N}$ exists globally and satisfies the  uniform estimates   \eqref{eq:smallcritic} and \eqref{eq:boundhigh}.
Proving  the convergence of that sequence will stem from  the following bounds for the differences between subsequent terms of the sequence.
\begin{prop}\label{stabro}
Let $(\sigma^n,u^n)_{n\in\N}$ be a sequence of approximate solutions. Then for 
 $\ds^n:=\sigma^{n}-\sigma^{n-1}$ and $\du^n:= u^{n}-u^{n-1}$, we have for all $n\geq1,$
$$\displaylines{\quad
\|\ds^n\|_{L^\infty_T \dot{B}^1_{d,1}}\leq C_5e^{C_5\|u^n\|_{L^1_T\dot{B}^2_{d,1}}}\Big[ \|\ds_0^n\|_{L^1_T\dot{B}^1_{d,1}} +\|\du^n\|_{L^1_T\dot{B}^2_{d,1}} \left(\|\sigma^{n-1}\|_{L^\infty_T\dot{B}^1_{d,1}}+1\right)\hfill\cr\hfill+ \|\nabla\sigma^{n-1}\|_{L^\infty_T\dot{B}^1_{d,1}} \|\du^n\|_{L^1_T\dot{B}^1_{d,1}}\Big],\quad}$$
and for all $n\geq2,$ 
$$\begin{aligned}
\|\du^n\|_{L^1_T\dot{B}^2_{d,1}} + \|\du^n\|_{L^\infty_T\dot{B}^{2-\alpha}_{d,1}}&\leq C_6\Big[\|\du^n_0\|_{\dot{B}^{2-\alpha}_{d,1}}+ \| \sigma^{n-2} \|_{L^\infty_T \dot B^1_{d,1}} \|  \du^{n-1} \|_{L^1_T \dot B^2_{d,1}} \\ 
& + \| \ds^{n-1} \|_{L^\infty_T \dot B^1_{d,1}} \| u^{n-1} \|_{L^1_T \dot B^2_{d,1}} +    \| \ds^{n-1} \|_{L^\infty_T \dot B^1_{d,1}} \| u^{n-2} \|_{L^1_T \dot B^2_{d,1}}\\
& +  \| \du^{n-1} \|_{L^1_T \dot B^2_{d,1}}  \|  \sigma^{n-1} \|_{L^\infty_T \dot B^1_{d,1}}   
 + \| u^{n-2} \|_{L^\infty_T \dot B^{2-\alpha}_{d,1}}  \| \du^{n-1} \|_{L^1_T \dot B^2_{d,1}}\\& + \| \du^{n-1} \|_{L^\infty_T \dot B^{2-\alpha}_{d,1}}  \|  u^{n-1} \|_{L^1_T \dot B^2_{d,1}}\Big]\cdotp
\end{aligned}$$
\end{prop}

\begin{proof}  In order to prove the first item, we use the fact that 
\begin{align*}
\ds^n_t + \div(u^n\ds^n) = -\div(\du^n\sigma^{n-1})-\div(\du^n).
\end{align*}
Then, applying  Proposition \ref{rubel} with 
\[\sigma = \ds^n,\quad u=u^n\andf f=-\div(\du^n\sigma^{n-1})-\div(\du^n),\]
and using the all too familiar argumentation based on Lemma \ref{super}, we get
$$\begin{aligned}
\|f\|_{L^1_T\dot{B}^1_{d,1}}&\leq C\Bigl(\|\sigma^{n-1}\div(\du^n)\|_{L^1_T\dot{B}^1_{d,1}} + \|\du^n\cdot\nabla\sigma^{n-1}\|_{L^1_T\dot{B}^1_{d,1}} 
+ \|\div\du^n\|_{L^1_T\dot{B}^1_{d,1}}\Bigr)\\
&\leq C\|\du^n\|_{L^1_T\dot{B}^2_{d,1}} \left(\|\sigma^{n-1}\|_{L^\infty_T\dot{B}^1_{d,1}}+1\right)
+ C\|\du^n\|_{L^1_T\dot{B}^1_{d,1}}\|\nabla\sigma^{n-1}\|_{L^\infty_T\dot{B}^1_{d,1}},
\end{aligned}$$
whence the desired inequality.
\medbreak
For proving  the second item, we observe from \eqref{un} that
\begin{multline*}
\partial_t \du^n  + \mu (-\Delta)^{\alpha/2} \du^n= 
- \sigma^{n-2} (- \Delta)^{\alpha/2} \du^{n-1} - \ds^{n-1} (- \Delta)^{\alpha /2}u^{n-1}\\ + I_\alpha (u^{n-2} , \ds^{n-1} ) 
+ I_\alpha ( \du^{n-1} , \sigma^{n-1})  - (u^{n-2} \cdot \nabla) \du^{n-1} - ( \du^{n-1} \cdot \nabla ) u^{n-1}.
\end{multline*}

Arguing  as in the previous steps, we get:
\begin{align*}
\|\sigma^{n-2} (- \Delta)^{\alpha/2} \du^{n-1} \|_{L^1_T \dot B^{2 - \alpha}_{d,1}} & \leq C   \| \du^{n-1} \|_{L^1_T \dot B^2_{d,1}} \| \sigma^{n-2} \|_{L^\infty_T \dot B^1_{d,1}} ; \\
\|\ds^{n-1} (- \Delta)^{\alpha /2} u^{n-1}\|_{L^1_T \dot B^{2 - \alpha}_{d,1}} & \leq C  \| u^{n-1} \|_{L^1_T \dot B^2_{d,1}}  \|  \ds^{n-1} \|_{L^\infty_T \dot B^1_{d,1}} ; \\
\|I_\alpha (u^{n-2} , \ds^{n-1} )\|_{L^1_T \dot B^{2 - \alpha}_{d,1}} & \leq C   \| u^{n-2} \|_{L^1_T \dot B^2_{d,1}} \| \ds^{n-1} \|_{L^\infty_T \dot B^1_{d,1}} ; \\
\|I_\alpha ( \du^{n-1} , \sigma^{n-1})\|_{L^1_T \dot B^{2 - \alpha}_{d,1}} & \leq C   \| \du^{n-1} \|_{L^1_T \dot B^2_{d,1}} \| \sigma^{n-1} \|_{L^\infty_T \dot B^1_{d,1}} ; \\
\|(u^{n-2} \cdot \nabla) \du^{n-1}\|_{L^1_T \dot B^{2 - \alpha}_{d,1}} & \leq C    \|  \du^{n-1} \|_{L^1_T \dot B^2_{d,1}} \| u^{n-2} \|_{L^\infty_T \dot B^{2-\alpha}_{d,1}} \ ;\\
\| ( \du^{n-1} \cdot \nabla ) u^{n-1} \|_{L^1_T \dot B^{2 - \alpha}_{d,1}} & \leq C \|\du^{n-1}\|_{L^\infty_T\dot{B}^{2-\alpha}_{d,1}} \|u^{n-1}\|_{L^1_T\dot{B}^2_{d,1}}\ .
\end{align*}
Then, taking advantage of  Proposition \ref{fracdiffestimates} completes  the proof of the second item.
\end{proof}

\subsection*{Step 6: The proof of existence}

In order to show that the sequence $(\sigma^n,u^n)_{n\in\N}$ converges, we are going to establish 
that, for all $T>0,$  it is a Cauchy sequence in the space
$$\cC([0,T];\dot B^1_{d,1})\times\bigl(\cC([0,T];\dot B^{2-\alpha}_{d,1})\cap L^1_T(\dot B^2_{d,1})\bigr)\cdotp$$
That property will come up from Proposition \ref{stabro} and Estimates \eqref{eq:smallcritic}--\eqref{eq:boundhigh} that entail
for all $n\geq1,$ changing slightly $C_5$ and $C_6$ and denoting by $C_0$ some 
constant that may  be computed from the right-hand side of  \eqref{eq:boundhigh},
\begin{eqnarray}
&&\|\ds^n\|_{L^\infty_T\dot{B}^1_{d,1}}\leq C_5  \|\ds^n_0\|_{\dot{B}^1_{d,1}} + 
C_5\|\du^n\|_{L^1_T\dot{B}^2_{d,1}}
+C_0\|\du^n\|_{L^1_T\dot{B}^1_{d,1}},\label{eq:ds}\\[2ex]
 &&\|\du^n\|_{L^1_T\dot{B}^2_{d,1}} + \|\du^n\|_{L^\infty_T\dot{B}^{2-\alpha}_{d,1}}
 \leq C_6\bigl(\|\du^n_0\|_{\dot{B}^{2-\alpha}_{d,1}}\nonumber\\
 &&\hspace{4cm}+ \eta \|\du^{n-1}\|_{L^1_T\dot{B}^2_{d,1}} + \eta \|\du^{n-1}\|_{L^\infty_T\dot{B}^{2-\alpha}_{d,1}}
 +\varepsilon    \|\ds^{n-1}\|_{L^\infty_T\dot{B}^1_{d,1}}\bigr)\cdotp\nonumber
 \end{eqnarray}

 Hence plugging the first inequality (at rank $n-1$) in the second one yields if $C_5\varepsilon\leq\eta,$
 $$\displaylines{
 \|\du^n\|_{L^1_T\dot{B}^2_{d,1}} + \|\du^n\|_{L^\infty_T\dot{B}^{2-\alpha}_{d,1}}
 \leq C_6\Bigl(\|\du^n_0\|_{\dot{B}^{2-\alpha}_{d,1}}+ C_5\varepsilon   \|\ds^n_0\|_{\dot{B}^1_{d,1}}\hfill\cr\hfill
 +2\eta( \|\du^{n-1}\|_{L^1_T\dot{B}^2_{d,1}} + \|\du^{n-1}\|_{L^\infty_T\dot{B}^{2-\alpha}_{d,1}})
 +C_0\varepsilon \|\du^{n-1}\|_{L^1_T\dot{B}^1_{d,1}}\Bigr)\cdotp} $$
 Using  interpolation, then Young inequality, we get that
 $$  \begin{aligned} C_0\|\du^{n-1}\|_{\dot{B}^1_{d,1}}&\leq C_0 \|\du^{n-1}\|_{\dot{B}^{2-\alpha}_{d,1}}^{\frac1\alpha}
 \|\du^{n-1}\|_{\dot{B}^2_{d,1}}^{\frac{\alpha-1}\alpha}\\
 &\leq\biggl(\frac{\alpha-1}\alpha\biggr)\|\du^{n-1}\|_{\dot B^2_{d,1}}+\frac{C_0^\alpha}\alpha  \|\du^{n-1}\|_{\dot{B}^{2-\alpha}_{d,1}}.\end{aligned}$$
 Therefore, denoting 
 $$
 \dU^n_T:= \|\du^n\|_{L^1_T\dot{B}^2_{d,1}} + \|\du^n\|_{L^\infty_T\dot{B}^{2-\alpha}_{d,1}}
 $$
 and assuming (just for expository purpose) that $C_5=C_6=1$, we end up with 
 $$
 \dU^n_T\leq \|\du^n_0\|_{\dot{B}^{2-\alpha}_{d,1}}+\varepsilon   \|\ds^n_0\|_{\dot{B}^1_{d,1}}
 + 2\eta\dU^{n-1}_T+ C_0^\alpha\varepsilon\int_0^T\dU^{n-1}_t\,dt.
 $$
 We sum over $n\geq2$ and get
 $$
  \sum_{n\geq2}  \dU^n_T\leq   \sum_{n\geq2}\bigl( \|\du^n_0\|_{\dot{B}^{2-\alpha}_{d,1}}+\varepsilon   \|\ds^n_0\|_{\dot{B}^1_{d,1}}\bigr)
  +\sum_{n\geq1} \biggl( 2\eta\dU^{n}_T+ C_0^\alpha\varepsilon\int_0^T\dU^{n}_t\,dt\biggr)\cdotp$$
 For small enough $\eta,$  bounding $\dU_T^1$ by means of \eqref{eq:smallcritic},  this implies that
 $$
  \sum_{n\geq1}  \dU^n_T\leq2\biggl(  
  \varepsilon+ \sum_{n\geq1}\bigl( \|\du^n_0\|_{\dot{B}^{2-\alpha}_{d,1}}+\varepsilon   \|\ds^n_0\|_{\dot{B}^1_{d,1}}\bigr)
 +C_0^\alpha \varepsilon \int_0^T \sum_{n\geq1}\dU^n_t\,dt\biggr)\cdotp$$
 Then, using Gronwall lemma and, we get
  $$
  \sum_{n\geq1}  \dU^n_T\leq  
 2\biggl(\varepsilon+ \sum_{n\geq1}\bigl( \|\du^n_0\|_{\dot{B}^{2-\alpha}_{d,1}}+\varepsilon   \|\ds^n_0\|_{\dot{B}^1_{d,1}}\bigr)\biggr)
 \exp\bigl(2C_0^\alpha\varepsilon T\bigr)\cdotp$$
The right-hand side being finite for all $T\geq0,$ 
  one may conclude that  $(u^n)_{n\in\N}$ is a Cauchy sequence in the Banach\footnote{See Lemma \ref{super}.} space 
$\cC([0,T];\dot{B}^{2-\alpha}_{d,1})\cap L^1_T\dot B^2_{d,1}$ for all $T\geq0.$
Then reverting to \eqref{eq:ds} and using a similar argument, we discover
that  $(\sigma^n)_{n\in\N}$ is a Cauchy sequence in the space $\cC([0,T];\dot{B}^{1}_{d,1}).$
So finally,  there exists a pair  $(\sigma, u)$ such that for all $T\geq0,$ we have
\begin{equation}\label{7a}
 (\sigma^n,u^n) \to (\sigma,u) \mbox{ in } \cC([0,T];\dot{B}^1_{d,1})\times \bigl(\cC([0,T];\dot{B}^{2-\alpha}_{d,1})\cap L^1_T\dot B^2_{p,1}\bigr)\cdotp
\end{equation}
The uniform estimates of the previous step and the properties of the Besov spaces for the weak convergence guarantee 
that we also have $(\nabla\sigma,\nabla u)$ in $L^\infty_T \dot{B}^1_{d,1}\times L^\infty_T\dot{B}^{2-\alpha}_{d,1}$
with the same bounds. 
\medbreak
Let us now check that  $(\sigma,u)$ indeed fulfills  \eqref{ren} in the sense of distributions. 
Regarding the continuity equation, we find that 
\begin{equation*}
 \sigma_t + \div (\sigma u) + \div u = (\sigma -\sigma^n)_t + \div ( \sigma u - \sigma^n u^n) + \div(u-u^n).
\end{equation*}
We have just proved that $\sigma_n\to \sigma$ in $\cC([0,T];\dot{B}^1_{d,1}),$ hence uniformly 
on $[0,T]\times\R^d.$ Therefore, $ (\sigma -\sigma^n)_t \to 0$ in the sense of distributions. 
To prove the convergence of  the last two terms,   we note that  the convergence property of \eqref{7a} for $u^n,$ 
the uniform estimates of the previous step and interpolation ensure that 
$$u^n\to u\ \hbox{ in }\ \cC([0,T];\dot B^1_{d,1}).$$
Hence, using once more the embedding of $\dot B^1_{d,1}$ in the set of continuous bounded functions, 
we see that $\sigma^n u^n$ converges uniformly to $\sigma u$ on $[0,T]\times\R^d.$
Therefore, we eventually have 
\begin{equation*}
 (\sigma -\sigma^n)_t + \div ( \sigma u - \sigma^n u^n) +\div(u-u^n) \to 0 \mbox{ in } \mathcal{D}'(\R_+ \times \R^d).
\end{equation*}

One can argue similarly for the momentum equation, writing that
\begin{multline*}
 u_t + u\cdot \nabla u + \mu(1+\sigma) (-\Delta)^{\alpha/2}u - I_\alpha(u,\sigma)
= (u-u^n)_t + (u\cdot \nabla u - u^n \cdot \nabla u^n)\\ + \mu\bigl((1+\sigma) (-\Delta)^{\alpha/2}u-(1+\sigma^n) (-\Delta )^{\alpha/2}u^n\bigr) - (I_\alpha(u,\sigma)-I_\alpha(u^n,\sigma^n)).
\end{multline*}
Making an extensive use of product estimates in Besov spaces, 
embedding and Lemma \ref{cruc}, one may  conclude that the right-hand side is going to zero in the distributional sense.
Hence $(\sigma,u)$ is a solution to System \eqref{ren}. 
\smallbreak
We still have to establish that  $(\nabla\sigma,\nabla u)$ is in $\cC(\R_+;\dot{B}^1_{d,1}\times \dot{B}^{2-\alpha}_{d,1})$ and  that 
$\nabla u$ is in  $L^1\dot B^2_{d,1}.$
The first property follows from the second one and  classical properties for the transport equation and
parabolic equations with fractional Laplacian. 

As regards the proof of $\nabla u\in L^1\dot B^2_{d,1},$
the difficulty is that having $(\nabla u^n)_{n\in\N}$  bounded in  $L^1\dot B^2_{d,1}$
just ensures that the weak limit $\nabla u$ is a measure on $\R_+$ with values on $\dot B^2_{d,1}.$ We follow ideas stated in \cite{DM777}.
In order to show that, indeed,  $\nabla u\in L^1\dot B^2_{d,1},$ one may use 
the fact that, owing to $\nabla u\in L^1\dot B^1_{d,1}$ and  Bernstein inequality, one may write  for all $J\in\N$ and  all   $n\in\N,$
$$\begin{aligned}
\sum_{j=-\infty}^J2^{2j}\int_{\R_+} \|\dot\Delta_j\nabla u\|_d\diff t
&\leq \sum_{j=-\infty}^J 2^{2j}\int_{\R_+}\|\dot\Delta_j\nabla u^n\|_d\diff t
+\sum_{j=-\infty}^J2^{2j}\int_{\R_+}\|\dot\Delta_j\nabla(u^n-u)\|_d\diff t\\
&\leq \|\nabla u^n\|_{L^1\dot B^2_{d,1}}+C2^J \sum_{j=-\infty}^J2^{2j}\int_{\R_+}\|\dot\Delta_j(u^n-u)\|_d\diff t.
\end{aligned}
$$
The first term in the right-hand side is uniformly bounded (see \eqref{eq:boundhigh})
while the last one tends to $0$ for $n$ going to $\infty.$
This completes the proof of the fact that $\nabla u$ belongs to $L^1\dot B^2_{d,1}$ and is bounded
by the right-hand side of \eqref{eq:boundhigh}.

\subsection*{Step 7: The case of a large initial velocity} 

We here  explain how the above arguments have to be adapted to handle large  initial 
velocity, assuming only  that   for some  small enough (absolute) $\eta>0,$
\begin{equation*}
\|\sigma_0\|_{\dot B^1_{d,1}}\leq\eta.
\end{equation*}
We keep the iterative scheme of Step 2 to define a sequence $(\sigma^n,u^n)_{n\in\N}$ of approximate solutions
(note that, there, no smallness is required whatsoever).
Then we denote
$$
U_0:=\sup_{n\in\N} \|u_0^n\|_{\dot B^{2-\alpha}_{d,1}},\quad
U'_0:=\sup_{n\in\N} \|\nabla u_0^n\|_{\dot B^{2-\alpha}_{d,1}}
 \andf
 S_0:=\sup_{n\in\N} \|\nabla\sigma_0^n\|_{\dot B^{1}_{d,1}}.
$$
We also introduce the notation
$$U^n_T:=\|u^n\|_{L_T^1\dot B^2_{d,1}}.$$
We claim that there exists a time $T>0$ (that will be bounded by below in terms
of $U_0$ and $U'_0,$ see below) so that for all $n\in\N,$ we have
\begin{equation}\label{eq:smalln}\begin{aligned}
&\|\sigma^n\|_{L_T^\infty\dot B^1_{d,1}}\leq2\eta,\qquad
\|u^n\|_{L^\infty_T\dot B^{2-\alpha}_{d,1}}+\|u^n\|_{L^1_T\dot B^{2}_{d,1}}\leq 2 U_0\\
\andf&\|\nabla u^n\|_{L^\infty_T\dot B^{2-\alpha}_{d,1}}+\|\nabla u^n\|_{L^1_T\dot B^{2}_{d,1}}
+\|\nabla\sigma^n\|_{L^\infty_T\dot B^{1}_{d,1}}\leq M(S_0+U'_0).
\end{aligned}\end{equation}
We shall argue by induction. The case $n=0$ being obvious, let us assume that \eqref{eq:smalln}
is true for $n-1$ and suppose that, for a small enough $c>0,$ 
\begin{equation}\label{eq:alpha}
U^{n-1}_T\leq c.\end{equation}
Then Inequality \eqref{ul1} tells us that 
$$\|u^n\|_{L^\infty_T\dot B^{2-\alpha}_{d,1}}+\|u^n\|_{L^1_T\dot B^{2}_{d,1}}\leq U_0
+2C_2U_T^{n-1}\bigl(\eta+ U_0\bigr)\cdotp$$
Hence \eqref{eq:smalln} is fulfilled by $u^n$ if
\begin{equation}\label{eq:cond0}
2 c C_2(\eta +U_0)\leq U_0.
\end{equation}
In order to bound the high norm of $u^n$, we shall slightly modify Proposition \ref{itermomentumdiffestimate}, 
estimating the term with $u^{n-1}\cdot\partial_k\nabla u^{n-1}$ as follows:
$$ \begin{aligned}
\| u^{n-1} \cdot   \partial_k\nabla u^{n-1} \|_{L^1_T \dot B^{2-\alpha}_{d,1}} & \leq C \int_0^T \| u^{n-1} \|_{\dot B^{1}_{d,1}} \| \partial_k\nabla u^{n-1} \|_{\dot B^{2-\alpha}_{d,1}}  \diff t \\
&\leq C \int_0^T  \| u^{n-1} \|_{\dot B^{2-\alpha}_{d,1}}^{\frac1\alpha} 
\| u^{n-1} \|_{\dot B^{2}_{d,1}}^{1-\frac1\alpha}
 \| \nabla u^{n-1} \|_{\dot B^{2}_{d,1}}^{\frac1\alpha}
  \| \nabla u^{n-1} \|_{\dot B^{2-\alpha}_{d,1}}^{1-\frac1\alpha} \diff t\\
  &\leq C(U^{n-1}_T)^{1-\frac1\alpha} \| u^{n-1} \|_{L_T^\infty\dot B^{2-\alpha}_{d,1}}^{\frac1\alpha} 
 \| \nabla u^{n-1} \|_{L_T^1\dot B^{2}_{d,1}}^{\frac1\alpha}
  \| \nabla u^{n-1} \|_{L_T^\infty\dot B^{2-\alpha}_{d,1}}^{1-\frac1\alpha}.\end{aligned} $$
  Bounding the other terms as in the proof of Proposition \ref{itermomentumdiffestimate} and using the first line of \eqref{eq:smalln} 
  at rank $n-1$ and \eqref{eq:alpha}, we end up with
\begin{multline}\label{eq:Dun}
\|\nabla u^n\|_{L^1_T \dot B^{2}_{d,1}\cap L^\infty_T \dot  B^{2-\alpha}_{d,1} } \leq C_4  \Big[ U'_0 +\eta \|\nabla u^{n-1}\|_{L^1_T \dot B^2_{d,1}}\\
+ c \bigl(\|\nabla \sigma^{n-1} \|_{L^\infty_T \dot B^1_{d,1}}+ \| \nabla u^{n-1} \|_{L^\infty_T \dot B^{2 -\alpha}_{d,1}}\bigr)  
+ (U^{n-1}_T)^{1-\frac1\alpha}  U_0^{\frac1\alpha} 
 \| \nabla u^{n-1} \|_{L_T^1\dot B^{2}_{d,1}}^{\frac1\alpha}
  \| \nabla u^{n-1} \|_{L_T^\infty\dot B^{2-\alpha}_{d,1}}^{1-\frac1\alpha}\Big]\cdotp\end{multline}
  Let us assume for a while that 
  \begin{equation}\label{eq:alphan}
U^{n}_T\leq c.\end{equation}
Then   Proposition \ref{nazwa} tells us that the first inequality of \eqref{eq:smalln} is fulfilled at rank $n$ if $c$ has been chosen so that
\begin{equation}\label{eq:cond1}
C_1 c\leq \log(3/2)\andf 3C_1c\leq\eta.
\end{equation} 
Then,  assuming also that $C_3c\leq\log2,$ Proposition \ref{nazwax} guarantees (increasing slightly $C_3$ if need be)  that 
$$\|\nabla\sigma^n\|_{L^\infty_T\dot B^1_{d,1}}\leq 2S_0+C_3\|\nabla u^n\|_{L^1_T\dot B^2_{d,1}}.$$
  At this stage, combining with \eqref{eq:Dun} and assuming also that $c\leq\eta,$ we discover that
$$\displaylines{
\frac12\|\nabla u^n\|_{L^1_T \dot B^{2}_{d,1}\cap L^\infty_T \dot  B^{2-\alpha}_{d,1} }+\frac1{2C_3}
\|\nabla\sigma^n\|_{L^\infty_T\dot B^1_{d,1}}\leq 2S_0
+ C_4  \Big[ U'_0 \hfill\cr\hfill
+ \eta \bigl(\|\nabla \sigma^{n-1} \|_{L^\infty_T \dot B^1_{d,1}}+ \| \nabla u^{n-1} \|_{L^1_T \dot B^{2}_{d,1}\cap L^\infty_T \dot B^{2 -\alpha}_{d,1}}\bigr)  
+ c^{1-\frac1\alpha}  U_0^{\frac1\alpha} 
 \| \nabla u^{n-1} \|_{L_T^1\dot B^{2}_{d,1}}^{\frac1\alpha}
  \| \nabla u^{n-1} \|_{L_T^\infty\dot B^{2-\alpha}_{d,1}}^{1-\frac1\alpha}\Big]\cdotp}
$$
  Since we assumed that \eqref{eq:smalln} is fulfilled at rank $n-1,$ we conclude that 
  $$\frac12\|\nabla u^n\|_{L^1_T \dot B^{2}_{d,1}\cap L^\infty_T \dot  B^{2-\alpha}_{d,1} }+\frac1{2C_3}
\|\nabla\sigma^n\|_{L^\infty_T\dot B^1_{d,1}}\leq 2S_0
+ C_4  \Big[ U'_0
+ (\eta M+ c^{1-\frac1\alpha}  U_0^{\frac1\alpha} )(S_0+U'_0)\Big]\cdotp
$$
Therefore, assuming with no loss of generality that $C_3\geq1$ and $C_4\geq2,$ 
  $$\|\nabla u^n\|_{L^1_T \dot B^{2}_{d,1}\cap L^\infty_T \dot  B^{2-\alpha}_{d,1} }+
\|\nabla\sigma^n\|_{L^\infty_T\dot B^1_{d,1}}\leq 2C_3C_4\Bigl[S_0
+  U'_0
+ (\eta M+ c^{1-\frac1\alpha}  U_0^{\frac1\alpha} )(S_0+U'_0)\Bigr]\cdotp
$$
Let us take $M:=4C_3C_4.$ Then the second line of \eqref{eq:smalln} is fulfilled provided 
\begin{equation*}
\eta M + c^{1-\frac1\alpha}  U_0^{\frac1\alpha} \leq1.
\end{equation*}
One can thus take $\eta \leq 1/(2M)$ (a constraint that does not depend on the size of $u_0$) and  $c$ fulfilling \eqref{eq:cond0}, \eqref{eq:cond1}
and $c^{1-\frac1\alpha}  U_0^{\frac1\alpha} \leq1/2.$ 
In order to complete the proof of uniform estimates,  we still have to justify  \eqref{eq:alphan}. 
Again, this stems from interpolation, as 
$$\begin{aligned}
U^n(T)&\leq C\int_0^T\|\nabla u^n\|_{\dot B^{2-\alpha}_{d,1}}^{\frac1\alpha}\|\nabla u^n\|_{\dot B^2_{d,1}}^{1-\frac1\alpha}\diff t\\
&\leq C T^{\frac1\alpha} \|\nabla u^n\|_{L_T^\infty\dot B^{2-\alpha}_{d,1}}^{\frac1\alpha}\|\nabla u^n\|_{L_T^1\dot B^2_{d,1}}^{1-\frac1\alpha}\diff t
\leq C MT^{\frac1\alpha} (S_0+U'_0).\end{aligned}
$$
Therefore, one can conclude to \eqref{eq:smalln}  provided $T$ fulfills
\begin{equation*}
\max\Bigl(U_0^{\frac1\alpha}\bigl(T^{\frac1\alpha} (S_0+U'_0)\bigr)^{1-\frac1\alpha},  T^{\frac1\alpha}(S_0+U'_0)\Bigr)\leq\varepsilon
\end{equation*}
for a small enough absolute $\varepsilon>0.$  
\medbreak
At this stage, one can easily repeat Step 6 so as to prove the convergence of the sequence $(\sigma^n,u^n)_{n\in\N}$
to some solution of \eqref{ren} fulfilling \eqref{eq:ET}. The details are left to the reader.


\subsection*{Step 8: Uniqueness}
In this step, we assume that we are given two solutions $(\sigma,u)$ and
$(\bar\sigma,\bar u)$ of \eqref{ren} on $[0,T]$ satisfying  \eqref{eq:ET}. 
One can assume that $(\sigma,u)$  is the one that has been constructed before, and thus, passing to the limit in \eqref{eq:smalln},
\begin{equation}\label{eq:smallsigma}
\|\sigma\|_{L_T^\infty\dot B^1_{d,1}}\leq2\eta.
\end{equation}
Proving uniqueness is essentially the same  as proving convergence, except
that we now consider the  difference  $(\ds,\du):=(\bar\sigma-\sigma, u-\bar u)$ between the  two solutions. 
Since we have 
$$\left\{\begin{aligned}
&\partial_t\ds+\div(\bar u\ds)=-\div\du-\div(\sigma\du),\\[1ex]
&\partial_t\du+\mu(-\Delta)^{\alpha/2}\du=I_\alpha(\du,\sigma)
+I_\alpha(\bar u,\ds) \\
&\hspace{3cm}-\mu\sigma(-\Delta)^{\alpha/2}\du-\mu\ds(-\Delta)^{\alpha/2}\bar u
-\bar u\cdot\nabla\du-\du\cdot\nabla u,\end{aligned}\right.
$$
applying Propositions \ref{rubel} and \ref{fracdiffestimates} and
the usual product laws in Besov spaces yields for all $t\in[0,T],$
$$\begin{aligned}& \|\ds(t)\|_{\dot B^1_{d,1}} \leq 
C\int_0^t \Bigl(\|\nabla\bar u\|_{\dot B^1_{d,1}}\|\ds\|_{\dot B^1_{d,1}}
+\|\div\du\|_{\dot B^1_{d,1}}\bigl(1+\|\sigma\|_{\dot B^1_{d,1}}\bigr) 
+\|\nabla\sigma\|_{\dot B^1_{d,1}} \|\du\|_{\dot B^1_{d,1}}\Bigr)\diff\tau,\\
& \|\du(t)\|_{\dot B^{2-\alpha}_{d,1}} + \|\du\|_{L^1_t\dot B^2_{d,1}}\leq 
 C\int_0^t\Bigl(\| \sigma \|_{\dot B^1_{d,1}}\|\du\|_{\dot B^2_{d,1}}+ \| \ds\|_{\dot B^1_{d,1}} \| \bar u \|_{\dot B^2_{d,1}} 
 \\&\hspace{7cm} + \|\bar u\|_{\dot B^{1}_{d,1}}  \|\nabla\du\|_{\dot B^{2-\alpha}_{d,1}}
 + \| \nabla u \|_{\dot B^{1}_{d,1}} \| \du \|_{\dot B^{2-\alpha}_{d,1}}\Bigr)\diff\tau.
 \end{aligned}$$
 Then, taking advantage of \eqref{eq:smallsigma}, we get for any small enough $c>0$ and $t\in[0,T],$
  \begin{multline}\label{eq:du}\|\du(t)\|_{\dot B^{2-\alpha}_{d,1}} + \|\du\|_{L^1_t\dot B^2_{d,1}}
+c  \|\ds(t)\|_{\dot B^1_{d,1}}
\\\leq C\int_0^t \Bigl(\| \ds\|_{\dot B^1_{d,1}} \| \bar u \|_{\dot B^2_{d,1}}
+c \|\du\|_{\dot B^1_{d,1}}\|\nabla\sigma\|_{\dot B^1_{d,1}} 
 + \|\bar u\|_{\dot B^{1}_{d,1}}  \|\du\|_{\dot B^{3-\alpha}_{d,1}}
  + \|  \nabla u \|_{\dot B^{1}_{d,1}} \| \du \|_{\dot B^{2-\alpha}_{d,1}}\Bigr)\diff\tau.\end{multline} 
 By interpolation and Young inequality, one may write
 $$\begin{aligned}
 cC \|\du\|_{\dot B^1_{d,1}}\|\nabla\sigma\|_{\dot B^1_{d,1}} &\leq
 cC\|\nabla\sigma\|_{\dot B^1_{d,1}} \|\du\|_{\dot B^{2-\alpha}_{d,1}}^{\frac1\alpha}
  \|\du\|_{\dot B^2_{d,1}}^{1-\frac1\alpha}\\
  &\leq \frac{\alpha-1}\alpha c^{\frac\alpha{\alpha-1}} \|\du\|_{\dot B^2_{d,1}}
  +\frac{C^\alpha}\alpha\|\nabla\sigma\|_{\dot B^1_{d,1}}^\alpha \|\du\|_{\dot B^{2-\alpha}_{d,1}},\end{aligned}
  $$
 and, similarly, 
 $$\begin{aligned}
 C \|\bar u \|_{\dot B^{1}_{d,1}} \| \du \|_{\dot B^{3-\alpha}_{d,1}}
  &\leq C\|\bar u\|_{\dot B^1_{d,1}}\|\du\|_{\dot B^2_{d,1}}^{\frac1\alpha}\|\du\|_{\dot B^{2-\alpha}_{d,1}}^{1-\frac1\alpha}\\
&\leq \frac1{\alpha} \|\du\|_{\dot B^2_{d,1}}
 +\frac{\alpha-1}\alpha (C\|\bar u\|_{\dot B^1_{d,1}})^{\frac{\alpha}{\alpha-1}}\|\du\|_{\dot B^{2-\alpha}_{d,1}}.
 \end{aligned}
 $$
 Hence taking $c$ small enough and plugging those two inequalities in \eqref{eq:du}, we get
 $$\displaylines{\quad
 \|\du(t)\|_{\dot B^{2-\alpha}_{d,1}} + \|\du\|_{L^1_t\dot B^2_{d,1}}
+c  \|\ds(t)\|_{\dot B^1_{d,1}}
\hfill\cr\hfill\leq C\int_0^t \Bigl(\| \ds\|_{\dot B^1_{d,1}} \| \bar u \|_{\dot B^2_{d,1}}
   + \bigl(\|\nabla\sigma\|_{\dot B^1_{d,1}}^\alpha+\|\bar u \|_{\dot B^{1}_{d,1}}^{\frac\alpha{\alpha-1}}+\|\nabla u \|_{\dot B^{1}_{d,1}}\bigr) \| \du \|_{\dot B^{2-\alpha}_{d,1}}\Bigr)\diff\tau.\quad}
 $$
 As having \eqref{eq:ET} implies that 
 $$\bar u\in L_T^\infty\dot B^{1}_{d,1}, \quad u,\bar u\in L_T^1\dot B^2_{d,1}\andf
 \nabla\sigma\in L_T^\infty\dot B^1_{d,1},$$
 applying   Gronwall lemma ensures that $(\ds,\du)\equiv0$ on $[0,T].$
 This implies uniqueness on the whole interval $[0,T].$


\appendix
\section{Proofs of Corollary \ref{coro1}, Propositions \ref{rubel} and \ref{fracdiffestimates}}

\begin{proof}[Proof of Corollary \ref{coro1}]
On the one hand, the fact that $u \in L^\infty(\R_+;\dot B^{2-\alpha}_{d,1}) \cap L^1 (\R_+;\dot B^2_{d,1})$
implies that   we have 
\begin{equation}\label{ax1}
 u\in L^m(\R_+;\dot B^1_{d,1}) \mbox{ \ \ \ with \ \ \ } m = \frac{\alpha}{\alpha-1}<\infty.
\end{equation}
On the other hand,     one can prove from the equation and the regularity of the solution given by Theorem \ref{main-intro} that
\begin{equation}\label{ax2}
 u_t \in L^1(\R_+;\dot B^{2-\alpha}_{d,1}\cap \dot B^{3-\alpha}_{d,1}) \subset L^1(\R_+;\dot B^1_{d,1}).
\end{equation}
Therefore, since $\dot B^1_{d,1}$ is embedded continuously in $L^\infty,$  we discover that for all $k\in \N,$   
\begin{equation*}
\|u\|_{L^1(k,k+1;L^\infty)}+\|u_t\|_{L^1(k,k+1;L^\infty)} \to 0 \mbox{ \ \ as \  \ } k \to \infty,
\end{equation*}
which readily gives the assertion of the corollary.
\end{proof}

\begin{proof}[Proof of Proposition \ref{rubel}]
The proof of existence begin standard,  we focus  on the estimates. Applying $\dot\Delta_j$ to \eqref{trans} yields
\begin{align*}
\partial_t\dot\Delta_j\sigma + \dot\Delta_j[\sigma{\rm div}u] + u\cdot\nabla\dot\Delta_j\sigma =  \dot\Delta_j f - R_j,
\end{align*}
where
\begin{align*}
R_j := u\cdot\nabla\dot\Delta_j\sigma - \dot\Delta_j ( u\cdot\nabla\sigma ) .
\end{align*}

Multiplying \eqref{trans} by $|\Delta_j\sigma|^{d-2}\Delta_j\sigma$ and integrating with respect to $x,$ we obtain
\begin{multline}\label{label}
 \frac{1}{d}\frac{d}{dt}\|\dot\Delta_j\sigma\|_d^d + \underbrace{\int_{\R^d}\dot\Delta_j[\sigma {\rm div} u]|\dot\Delta_j\sigma|^{d-2}\dot\Delta_j\sigma \diff x}_{I_1} + \underbrace{\int_{\R^d} u\cdot\nabla\dot\Delta_j \sigma |\dot\Delta_j\sigma|^{d-2}\dot\Delta_j\sigma \diff x}_{I_2} \\ = \int_{\R^d} \dot\Delta_jf|\dot\Delta_j\sigma|^{d-2}\dot\Delta_j\sigma \diff x - \underbrace{\int_{\R^d}R_j|\dot\Delta_j\sigma|^{d-2}\dot\Delta_j\sigma\diff x}_{I_3} .
\end{multline}
By H\" older's inequality, we have
$$|I_1|\leq \|\dot\Delta_j[\sigma{\rm div}]u\|_d\|\dot\Delta_j\sigma\|_d^{d-1}.$$
By integration by parts (note that $\nabla\sigma$ is a Schwartz function)
$$
I_2 = \frac{1}{d}\int_{\R^d} u\cdot\nabla|\dot\Delta_j \sigma|^d\diff x = -\frac{1}{d}\int_{\R^d}{\rm div} u \,|\dot\Delta_j\sigma|^d\diff x,
$$
which, since $\dot{B}^{1}_{d,1}\hookrightarrow L^\infty$ by Lemma \ref{super}, implies that
$$
|I_2|\leq C \|{\rm div} u\|_{\dot{B}^{1}_{d,1}}\|\dot\Delta_j\sigma\|_d^d.
$$
Regarding the commutator term $R_j$ in $I_3$, we have by H\" older's inequality
$$|I_3|\leq \|R_j\|_d\|\dot\Delta_j\sigma\|_d^{d-1}.$$
Combining the above estimates of $I_1, I_2$ and $I_3$ with \eqref{label}, we end up with
\begin{align*}
\frac{1}{d}\frac{d}{dt}\|\dot\Delta_j\sigma\|_d^d\leq C\|\nabla u\|_{\dot{B}^{1}_{d,1}}\|\dot\Delta_j\sigma\|_d^d 
+ \left(\|\dot\Delta_j[\sigma {\rm div}u]\|_d + \|R_j\|_d+\|\dot\Delta_jf\|_d\right)\|\dot\Delta_j\sigma\|_d^{d-1}
\end{align*}
which, after time integration, leads to 
$$\|\dot\Delta_j\sigma(t)\|_d\leq \|\dot\Delta_j\sigma_0\|_d +\int_0^t\|\dot\Delta_j f\|_d\diff\tau
+  \int_0^t\Bigl( C\|\nabla u\|_{\dot{B}^{1}_{d,1}}\|\dot\Delta_j\sigma\|_d+ \|\dot\Delta_j[\sigma {\rm div}u]\|_d + \|R_j\|_d\Bigr)\diff\tau.$$
Then, multiplying by $2^{j}$ and summing over $j\in{\mathbb Z}$, we obtain for all $t\geq0,$
$$
\|\sigma(t)\|_{\dot{B}^1_{d,1}}\leq \|\sigma_0\|_{\dot{B}^1_{d,1}} +\int_0^t\biggl(\|f\|_{\dot B^1_{d,1}}+\|\sigma  {\rm div}u\|_{\dot B^1_{d,1}}
+C\|\nabla u\|_{\dot B^1_{d,1}}\|\sigma\|_{\dot B^1_{d,1}}
+\sum_{j\in\Z}2^j\|R_j\|_d\,\diff\tau\biggr)\cdotp$$
By Lemma \ref{super}\ref{prod}, we have 
$$\|\sigma  {\rm div}u\|_{\dot B^1_{d,1}}\leq \|\sigma\|_{\dot B^1_{d,1}}\|{\rm div}u\|_{\dot B^1_{d,1}},$$
and by Lemma \ref{super}\ref{commutator} we know that
$$\sum_{j\in\Z}2^j\|R_j\|_d = \| R_j \|_{\dot B^1_{d,1}} \leq C\|\nabla u\|_{\dot{B}^{1}_{d,1}}\|\sigma\|_{\dot{B}^1_{d,1}}.$$
Hence, altogether, we get
the desired inequality. 
\end{proof}


Proposition \ref{fracdiffestimates} relies on 
 the following Bernstein-type lemma for the fractional Laplacian semigroup.

\begin{lem}
Let $\cC$ be an annulus centered at $0,$ and $\alpha>0$. There exist $C,c >0$ such that for any $1 \leq p \leq \infty$ and $\lambda >0$,  if $\supp \mathcal F u \subset \lambda \cC$ then 
\begin{equation} \label{bernstein}
\left\| e^{-t ((- \Delta)^{\alpha / 2})} u \right\|_{p} \leq C e^{- c\lambda^\alpha t} \| u \|_{p} \quad\hbox{for all }\ t\geq0.
\end{equation}
\end{lem}
\begin{proof}
The proof can be found in e.g. \cite[Proposition 2.2]{hmidi} or  \cite[Lemma 3.1]{wu}.
\end{proof}

\begin{proof}[Proof of Proposition \ref{fracdiffestimates}]
Arguing by density, it suffices to consider the case where  $u_0$  and $f$ are 
in $\cS_0$ and $\cC(\R_+;\cS_0),$ respectively. 
Then  the Cauchy problem \eqref{fracdiff} has a unique solution $u$ in $\cC^1(\R_+;\cS_0)$ that  satisfies 
\begin{equation*}
\widehat u(t,\xi) = e^{-t \mu | \xi|^{\alpha}} \widehat{u}_0 (\xi)+ \int_0^t e^{- \mu |\xi|^\alpha(t-\tau)}  \widehat f(\tau,\xi)\diff\tau.
\end{equation*}
We take  the Fourier transform of \eqref{fracdiff} with respect to the space variable,
 and  multiply it by the  function $\varphi(2^{-k} \ \cdot )$ from the Littlewood-Paley decomposition, obtaining
\begin{equation*}
 \varphi(2^{-k} \ \cdot) \widehat u_t + \mu |\xi|^\alpha \varphi(2^{-k}\ \cdot) \widehat u = \varphi(2^{-k}\ \cdot)  \widehat f \ .
\end{equation*}
Using estimate \eqref{bernstein} and Minkowski inequality, we obtain the following time-pointwise estimate:
\begin{equation*}
 \|\dot \Delta_k u (t) \|_{L^p}  \leq C e^{- c\mu t  2^{\alpha k}} \|\dot \Delta_k u_0\|_{p} + C \int_0^t e^{- c \mu 2^{k\alpha}(t-\tau)} \|\dot \Delta_k f(\tau)\|_{p}\diff\tau.
\end{equation*}
Hence 
\begin{equation*}
\begin{split}
\|u\|_{L^\infty_T \dot B^s_{p,1} } =\sup_{t \in (0,T)} \left\{ \sum_{k\in \Z} 2^{s k}  \|\dot \Delta_k u\|_{p} \right\}
\leq C\|u_0\|_{\dot B^s_{p,1} } + C\int_0^T \|f(t)\|_{\dot B^s_{p,1}}\diff t
\end{split}
\end{equation*}
and 
$$\begin{aligned}
\|u\|_{L^1_T \dot B^{s+\alpha}_{p,1}} &= \int_0^T \sum_{k\in\Z} 2^{k(\alpha+s)} \|\Delta_k u(t)\|_p\diff t \\
&\leq C\int_0^T \sum_{k\in\Z} 2^{k(\alpha+s)} \left\{ e^{-ct\mu 2^{\alpha k}} \|\Delta_k u_0\|_p  +  \int_0^t e^{-c\mu2^{\alpha k}(t-\tau)} \|\Delta_k f(\tau)\|_p\,d\tau   \right\} \diff t \\
&\leq C \sum_{k\in\Z} \biggl(2^{k(\alpha+s)} 2^{-\alpha k} \|\Delta_k u_0\|_p +  2^{k(\alpha +s)} \|e^{-c\mu2^{\alpha k}t}\|_{L^1(0,T)} \int_0^T  \|\Delta_k f(t)\|_{p}\diff t\biggr)  \\
& \leq  C \|u_0\|_{\dot B^s_{p,1}} + C \sum_{k\in\Z} 2^{k(\alpha +s)} 2^{-\alpha k}  \int_0^T  \|\Delta_k f(t)\|_{p}\diff t \\
&\leq C\Bigl( \|u_0\|_{\dot B^s_{p,1}} +   \|f\|_{L^1_T \dot B^s_{p,1}} \Bigr)\cdotp
\end{aligned}
$$\end{proof}

%
%
%

\noindent{\bf Acknowledgments.}
The first author (R.D.)  is partly supported by ANR-15-CE40-0011.
The second and third authors (P.B.M. \& J.P.) have been partly supported by National  Science  Centre  grant 2014/14/M/ST1/00108 (Harmonia).
The third author was supported by the Polish MNiSW grant Mobilno\' s\' c Plus no. 1617/MOB/V/2017/0.
The fourth author (B.W.) was supported by the Polish NCN grant no. 2016/23/B/ST1/00434.
This work was partially supported by the Simons-Foundation grant 346300 and the Polish Government
MNiSW 2015–2019 matching fund.

\bibliographystyle{abbrv}
\bibliography{CS-16bis}

\end{document}